\documentclass{amsart}[12pt]

\usepackage{fullpage}
\usepackage{amsmath}
\usepackage{amsfonts}
\usepackage{amssymb}
\usepackage[pdftex]{graphicx}

\numberwithin{equation}{section}

\newtheorem{theorem}{Theorem}
\newtheorem*{theorem*}{Theorem}
\newtheorem{lemma}{Lemma}

\theoremstyle{definition}

\newtheorem*{remark*}{Remark}
\newtheorem{remark}{Remark}

\begin{document}
\title{Counting functions for sums of rational powers of integers}

\author{Trevor Wine}

\begin{abstract}
Counting functions are constructed for sums of integers raised to a fixed positive rational power. That is, given values formed by $u_1^{j/k} + u_2^{j/k} + ... + u_l^{j/k}$, $u_i \in \mathbb{Z}^+$, the number of values less than or equal to a given $w>0$ is determined. The counting functions developed are framed in terms of convolution exponentials, and are closely related to the Riemann zeta function. At the conclusion, several estimates are derived, with special emphasis on the case of sums of square roots, i.e. $j=1$, $k=2$.
\end{abstract}

\maketitle

\section{\bf Introduction}

This paper develops counting functions for sums of positive rational powers, $j/k$ ($k>1$, $j/k$ in lowest terms), of positive integers. In other words, given the set $\mathbb{M}_{j,k} \equiv \{ u_1^{j/k} + u_2^{j/k} + ... + u_l^{j/k}  | u \in \mathbb{Z}^+ \}$, the number of elements at or below a given $w$ is determined (it should be noted, as the set notation implies, the values in $\mathbb{M}$ are unique--for example, $\sqrt{12}+\sqrt{48}$ and $\sqrt{3}+\sqrt{75}$ are counted only once). More formally, we study the class of functions in $j$ and $k$ defined by,
\begin{equation} \label{Sjk}
S_{j,k}(w) = \#\{v \in \mathbb{M}_{j,k} | v \le w\}.
\end{equation}

The main result of the paper is the following theorem for $S_{j,k}$:
\begin{theorem} \label{S_jk_theorem}
\[
S_{j,k}(w) = \int_{t=0^+}^{w} \exp^*(dI_{j,k}) dt,
\]
where $\exp^*$ denotes the convolution exponential, and where the exponential power function, $I_{j,k}$, is given in terms of the Riemann zeta function and its zeros,
\begin{equation} \begin{split} \nonumber
I_{j,k}(w) = & \frac{\zeta(1+k/j)}{\zeta(k)} w^{k/j} +  \lim_{m\to\infty}\stackrel{*}{\sum_{|\mbox{Im }\rho|<Y_m/(jk)}} \frac{\zeta(\rho/j+1)\zeta(\rho/k)} {\zeta ^\prime (\rho)}\frac{{ w}^{\rho/j}}{\rho} + \ln w \\
  & + \gamma - j\left( 1- \frac{1}{k} \right) \ln 2\pi + \mathcal{O}\left(w^{-3/(2j)+\epsilon}\right),
\end{split} \end{equation}
($\rho$ are the non-trivial zeros of the Riemann zeta function, $\zeta$; * indicates this form for the summands is valid for simple zeta zeros only; the $Y_m$ are specifically chosen to ensure convergence of the sum, as will be shown; $\gamma$ is the Euler-Mascheroni constant).

The sum over the zeta zeros and the lower order terms may be bounded as follows,
\[
I_{j,k}(w) = \frac{\zeta(1+k/j)}{\zeta(k)} w^{k/j} + \mathcal{O}(w^{1/j}),
\]
(where this bound may be improved, conditional on the Riemann hypothesis).
\end{theorem}
The theorem will be proved in sections \ref{remainder_integral_summary_section} and \ref{I_error_bound_section}, after the necessary groundwork.

The main approach of the paper is to borrow the Euler product form for the Riemann zeta function, and substitute a set of relevant `additive primes', for the usual multiplicative primes. Specifically, the counting functions, $S_{j,k}(w)$, are built based on an analog to the familiar Riemann zeta formulation as an Euler product,
\begin{equation} \label{zeta_1}
\zeta(s) = \sum_n n^{-s} = \prod_p \left( 1-p^{-s} \right)^{-1} \mbox{ Re $s > 1$},
\end{equation}
namely,
\begin{equation} \label{Z_jk}
Z_{j,k}(s) = 1+\sum_{v \in \mathbb{M}_{j,k}}  e^{-s v} = \prod_{i} \left( 1-e^{-s n_i^{j/k}} \right)^{-1},
\end{equation}

where the $n_i$ are k-free integers, and region of convergence will be determined below. The analog to the product form follows since just as any integer can be written uniquely as $p_{1}^{j_1}...p_{t}^{j_t}$, the $p_{i}$ multiplicative primes, so any value in $\mathbb{M}_{j,k}$ can be written uniquely as $m_1 n_1^{j/k} +... + m_l n_l^{j/k}$, $m_i \in \mathbb{Z}^+$, where each $n_i$ is a \emph{unique} k-free number; the set $\{ n_i^{j/k} \}_i$ constitutes the desired `additive primes'. To help see the legitimacy of this unique `prime factorization', consider that for any $u_1^{j/k} + u_2^{j/k} + ... + u_l^{j/k}$, each $u_i$ can be uniquely factored into a k-free, and a `k-full' portion. Repeated addition among like k-free terms (certainly at least when $u_i$ itself is k-free) can then yield any integer multiple of a given $n_i^{j/k}$.

\section{\bf $\mathbf{S_{j,k}(w)}$ as a convolution exponential}

\subsection{$\mathbf{Z_{j,k}(s)}$ product and sum forms.}

  To develop $Z_{j,k}$, both the product and sum forms in $(\ref{Z_jk})$ will be shown to converge to the same analytic function in the region $H \equiv \{ s | \mbox{Re }s = a > 0\}$. 

  To examine the product form first, a basic theorem on infinite products (see \cite{Markushevich}, vol. 1) will be useful:

\begin{theorem*}
If every term of the sequence of functions $\{v_i(s)\}$ is analytic on domain G, and if $v_i(s) \ne 1$ for all $s$ in G, and if there exists a convergent series $\sum_n M_n$ such that $|v_i(s)| \le M_n$ for all $s$ in G, then the product,
\[
\prod_{i} \left( 1-v_i(s) \right)^{-1}
\]
converges uniformly on G to a nonvanishing analytic function $f(s)$.
\end{theorem*}

Application of the theorem to the product form for $Z_{j,k}(s)$, 
\begin{equation} \label{Z_product_form}
\prod_{i} \left( 1-e^{-s n_i^{j/k}} \right)^{-1},
\end{equation}
follows easily when the bounding $M_n$ are applied over the domain $H_{\delta}\equiv \{ s | \mbox{Re }s = a \ge \delta > 0 \}$,  as $M_n = e^{-\delta n^{j/k}}$, producing
\begin{equation} \label{ln_Z_jk_rough_bound}
\sum_n M_n < \int_{x=0}^{\infty} e^{-\delta x^{j/k}} dx = \frac{k}{j} \frac{1}{\delta^{k/j}} \int_{y=0}^{\infty} y^{k/j-1}e^{-y} dy = \frac{k}{j} \Gamma \left( \frac{k}{j} \right) \frac{1}{\delta^{k/j}} ,
\end{equation}
and completing the requirements of the theorem. The product form therefore converges to an analytic function in the $H$ half-plane and does so uniformly in any $H_{\delta}$.

To prove convergence of the sum form of $Z_{j,k}(s)$, consider $s \in \mathbb{R} \bigcap H = x$, and notice
\[
1+\sum_{v \in \mathbb{M}_{j,k}, v<N}  e^{-x v} < \prod_{i=1}^{M} \left( 1-e^{-x n_i^{j/k}} \right)^{-1} < \prod_{i=1}^{\infty} \left( 1-e^{-x n_i^{j/k}} \right)^{-1} < B,
\]
where the leftmost inequality follows by choosing $M$ large enough in the product to cover all $v < N$ in the sum, and the rightmost inequality follows from convergence of the product. The infinite sum, being monotone and bounded, converges. Then from a well known result for Dirichlet series, if the series converges at $\mbox{Re}(s) = x$ then it converges in the half-plane $\{s|\mbox{Re}(s)>x\}$ (see \cite{Markushevich}, vol. 2). This proves convergence of the sum in $(\ref{Z_jk})$ in half-plane $H$, and by analytic continuation equals the same analytic function as product form $(\ref{Z_product_form})$.

\subsection{\bf Initial construction of $\mathbf{S_{j,k}}$; deriving a form for $\mathbf{\ln Z(s)}$.}

Applying Perron's formula to the Dirichlet series, $Z_{j,k}(s)$ from above, the main formula for the counting functions $(\ref{Sjk})$ will be,
\begin{equation} \label{S_integral_1}
S_{j,k}(w) = \frac{1}{2\pi i} \int_{s=a-i\infty}^{a+i\infty} e^{sw} \frac{e^{\ln Z(s)}}{s} ds -1,
\end{equation}
where $\mbox{Re } s > 0$, and with the important restrictions that integer powers are not allowed, and that the fraction $j/k$ must be in lowest form. The $-1$ correction factor prevents counting the extraneous $1$ in the sum form of $(\ref{Z_jk})$, the value $v=0$ not being an element of $\mathbb{M}_{j,k}$. (For details on applying Perron's formula to Dirichlet series to extract counting functions, see for example \cite{Hardy}.) Note also that $e^{\ln Z(s)}$ has been used instead of $Z(s)$ in the integrand of $(\ref{S_integral_1})$. This is because, as we will see, a convenient expression for $\ln Z(s)$ is much easier to derive than one for $Z(s)$.

Before developing an expression for $\ln Z_{j,k}(s)$, two of its sum forms will be shown. The first follows from the Euler product form for $Z_{j,k}(s)$ of equation $(\ref{Z_jk})$. We get, on applying the logarithm,
\[
\ln Z_{j,k}(s) = \sum_i -\ln(1-e^{-s n_i^{j/k}}) = \sum_i \left( e^{-s n_i^{j/k}} + \frac{e^{-2s n_i^{j/k}}}{2} + \frac{e^{-3s n_i^{j/k}}}{3} + ... \right)
\]
\begin{equation} \label{ln_Z_jk_sum_form}
= \sum_{m,i} \frac{1}{m} e^{-smn_i^{j/k}}
\end{equation}
(where the condition to assure convergence of the Taylor expansion of the logarithm, $|e^{-s n_i^{j/k}}|<1$, holds everywhere in the $\mbox{Re}^+$ half-plane $H$). 

Another sum form for $(\ref{ln_Z_jk_sum_form})$ will be useful:
\begin{equation} \label{ln_Z_sum_form}
= \sum_{l=1}^{\infty} a_l e^{-s \sqrt[k]{l}},
\end{equation}
where 
\[
a_l = 
  \left\{
    \begin{array}{ll}
	0 & \mbox{ if $l\ne m^k n_i^j$ }, \\
	\frac{1}{m} & \mbox{ if $l = m^k n_i^j$},
    \end{array}
  \right.
\]
In both sums, $m \in \mathbb{Z}^+$, and, as above, the $n_i$ index the k-free integers.

To develop an analytic expression for $\ln Z(s)$, denoting $Q_k(t)$ as the k-free counting function, and $dQ_k$ as having point masses of weight $1$ at each k-free integer, we have the following Stieltjes integral,
\begin{equation} \label{ln_Z_dQ_integral}
\ln Z_{j,k}(s) = -\int_{t=0}^{\infty} \ln(1-e^{-st^{j/k}})dQ_k(t).
\end{equation}

Next, in a well-known result the k-free counting function may be written in terms of the Riemann zeta function,
\[
Q_k(t) = \frac{1}{2\pi i} \int_{z=x-i\infty}^{x+i\infty} \frac{\zeta(z)t^z}{\zeta(kz) z} dz 
\]
when $\mbox{Re }z=x>1$. (This follows from an application of Perron's formula to,
\begin{equation} \label{zeta_square_free}
\sum_n \frac{\mu_k(n)}{n^s} = \frac{\zeta(s)}{\zeta(ks)},
\end{equation}
where $\mu_k(n)$ denotes the characteristic function for the k-free integers (not to be confused with the (signed) Mobius function), as may be readily derived by writing each zeta function in its product form.) Equation $(\ref{ln_Z_dQ_integral})$ may then be integrated by parts (noting $Q_k(t)=0$ when $t<1$ and $\ln(1-e^{-s t^{j/k}}) Q_k(t) \rightarrow 0$ as $t\rightarrow \infty$). Substituting $u = t^{j/k}$ we  then have for have for $(\ref{ln_Z_dQ_integral})$,
\begin{equation} \label{ln_Z_jk_pre_swap}
= \int_{u=0}^{\infty}  \frac{s}{2 \pi i} \int_{z=x-i\infty}^{x+i\infty}   \frac{u^{(k/j)z}}{e^{su}-1} \frac{\zeta(z)}{\zeta(kz) z} dz du.
\end{equation}

To obtain an expression for $\ln Z(s)$, we'll swap the order of integration. The swap will be proved by first assuming it is true, then confirming the result equals one of the sum forms for $\ln Z_{j,k}(s)$ just above. This approach has the advantage of better illustrating the role of the resulting tri-zeta term in the integrand, which is central to both $\ln Z_{j,k}(s)$ and the eventual integral form for $I_{j,k}(w)$.

Swapping produces,
\begin{equation} \label{ln_Z_jk_post_swap}
\frac{s}{2 \pi i} \int_{z=x-i\infty}^{x+i\infty} \frac{\zeta(z)}{\zeta(kz) z} \int_{u=0}^{\infty} \frac{u^{(k/j)z}}{e^{su}-1} du dz = \frac{1}{2\pi i} \int_{z=x-i\infty}^{x+i\infty} \frac{\zeta(z)}{\zeta(kz) z} \frac{1}{s^{(k/j)z}} \int_{v=0}^{s\infty} \frac{v^{(k/j)z}}{e^v-1} dv dz
\end{equation}
after the substitution $v = su$. Now by common results in Mellin transforms (see, for example, \cite{Edwards}),
\begin{equation} \label{zeta_gamma_trick}
\int_0^{\infty} \frac{y^{\beta-1}}{e^y - 1} dy = \Gamma(\beta)\sum_{n=1}^{\infty} \frac{1}{n^{\beta}} = \Gamma(\beta)\zeta(\beta)
\end{equation}
where $\mbox{Re }\beta$ must be greater than $1$. Noting the integrand of $(\ref{ln_Z_jk_post_swap})$ is bounded as
\[
\left| \frac{v^{(k/j)z}}{e^v-1} \right| < c_0 \frac{ \left(|s|R\right)^{(k/j)x} e^{(k/j)(\pi/2)|y|} }{ e^{aR} }
\]
along the arc from $Rs$ to $R|s|$ where $a>0$, $R >> 0$ and $v = re^{i\theta}$, the contour integral along this arc goes to zero as $R\rightarrow \infty$, producing by Cauchy's theorem,
\[
\int_{v=0}^{s\infty} \frac{v^{(k/j)z}}{e^v-1} dv = \Gamma((k/j)z+1) \zeta((k/j)z+1).
\]
Putting it all together, with the change of variables $z/j \rightarrow z$, we have,
\begin{equation} \label{ln_Z_jk_tri_zeta}
\frac{1}{2\pi i} \int_{z=x-i\infty}^{x+i\infty} \frac{\zeta(kz+1)\zeta(jz)}{\zeta(jkz)}\frac{k\Gamma(kz)}{s^{kz}}dz 
\end{equation}
where $\mbox{Re }z >1/j$ and $\mbox{Re }s > 0$.

To now show the swap was valid and that $(\ref{ln_Z_jk_tri_zeta})$ in fact equals $\ln Z_{j,k}(s)$, notice the tri-zeta term can be expanded as a double series,
\[
\frac { \zeta(kz+1) \zeta(jz) }{ \zeta(jkz) } = \left( \sum_{n=1}^{\infty} \frac{ \mu_k(n) }{ n^{jz} } \right) \left( \sum_{m=1}^{\infty} \frac{1/m}{m^{kz}} \right).
\]

Since both series are absolutely convergent for $\mbox{Re}(z)>1/j$, distributivity produces a convergent double series:
\begin{equation} \label{tri_zeta_positive}
\frac { \zeta(kz+1) \zeta(jz) }{ \zeta(jkz) } = \sum_{m,i} \frac{1}{m} \left( m^k n_i^j \right)^{-z} = \sum_l a_l l^{-z},
\end{equation}
where as usual, the $n_i$ index the k-frees, and $a_l$ is defined as in equation $(\ref{ln_Z_sum_form})$. This produces, for $(\ref{ln_Z_jk_tri_zeta})$,
\[
=\frac{1}{2\pi i} \int_{z=x-i\infty}^{x+i\infty} \left( \sum_{l=1}^{\infty} a_l l^{-z} \right) \frac{k\Gamma(kz)}{s^{kz}}dz.
\]
This can readily be evaluated by swapping the order of summation and integration, the swap justified by the dominated convergence theorem. (Note that
\[
\sum_{l=1}^{\infty} \int_{z=x-i\infty}^{x+i\infty} \left| a_l l^{-z} \frac{k\Gamma(kz)}{s^{kz}} \right| dz = k \sum_l a_l l^{-x} \int_z \left|\frac{\Gamma(kz)}{s^{kz}} \right|dz < \infty,
\]
where the final inequality follows from the convergence of the gamma function integral, noting that if $s=re^{i\theta}$ and $z=qe^{i\phi}$, then $|s^{-(k/j)z}| < c_1 e^{(k/j)y(\pi/2-\epsilon)}$ and $\Gamma((k/j)z) < c_2e^{-(k/j)y\pi/2}$, where we've assumed $|\arg \{z,s\}|<\pi$ for the estimates, and since $\mbox{Re }s>0$ means $|\theta|<\pi/2$, while $|\phi|\rightarrow \pi/2$ in the limit of large $|y|$. The dominated convergence theorem now may be applied to the partial sums in $l$, using $g(z)=\sum_l k a_l l^{-x} |\Gamma(kz) s^{-kz}|$ as the dominating function.)

The post-swapped form then is,
\[
\sum_{l=1}^{\infty} a_l \frac{1}{2\pi i} \int_{z=x-i\infty}^{x+i\infty}  k\Gamma(kz)(s\sqrt[k]{l})^{-kz}dz.
\]

By common results on Mellin transforms (see for example \cite{G_R}), provided $\mbox{Re }s>0$,
\[
e^{-s} = \frac{1}{2\pi i} \int_{z=x-i\infty}^{x+i\infty} \Gamma(z) s^{-z}dz,
\]
producing for each summand integral,
\[
a_l \frac{1}{2\pi i} \int_{z=x-i\infty}^{x+i\infty}  k\Gamma(kz)(s\sqrt[k]{l})^{-kz}dz = a_l e^{-s\sqrt[k]{l}}
\]
producing the series, $\sum_l a_l e^{-s\sqrt[k]{l}}$. Compare with $(\ref{ln_Z_sum_form})$. This proves $(\ref{ln_Z_jk_tri_zeta})$ is a valid expression for $\ln Z_{j,k}(s)$.

\subsection{\bf $\mathbf{I_{j,k}(w)}$ introduction and its formulation in terms of the Riemann zeta function.}

Define the function $I_{j,k}$ as the step function formed by Perron's formula applied to $\ln Z_{j,k}(s)$,
\begin{equation} \label{I_function_ln_Z}
I_{j,k}(w) = \frac{1}{2\pi i} \int_{s=a-i\infty}^{a+i\infty} \frac{e^{sw}}{s} \ln Z_{j,k}(s) ds,
\end{equation}
valid for $\mbox{Re }s > 0$, $w>0$.

Combining this with the first sum form for $\ln Z_{j,k}(s)$, $(\ref{ln_Z_jk_sum_form})$, $I_{j,k}(w)$ is a step function with value $0$ at $w=0$, with jumps of height $1$ at every $n_i^{j/k}$, $n_i$ a k-free, and in general of height $\frac{1}{m}$ at every $m n_i^{j/k}$, and where its value at each jump is found by averaging the height of the step just above and the step just below--that is in general,
\begin{equation} \label{I_steps}
I_{j,k}(w) = \frac{1}{2} \left( \sum_{m n_i^{j/k} < w} \frac{1}{m}  + \sum_{m n_i^{j/k} \le w} \frac{1}{m} \right)
\end{equation}

Now to develop the expression for $I_{j,k}$, combine $(\ref{I_function_ln_Z})$ with the integral form for $\ln Z_{j,k}(s)$ in ($\ref{ln_Z_jk_tri_zeta}$) to give,
\begin{equation} \label{I_pre_swap}
I_{j,k}(w)= \frac{1}{2\pi i} \int_{s=a-i\infty}^{a+i\infty} \frac{1}{2\pi i} \int_{z=x-i\infty}^{x+i\infty} \frac{e^{s w}}{s} \frac{\zeta(kz+1)\zeta(jz)}{\zeta(jkz)}\frac{k\Gamma(kz)}{s^{kz}} dz ds,
\end{equation}
where $a>0$ and $x>1/j$. To evaluate this double integral, we'll swap the order of integration, using the same approach as the double integral for $\ln Z_{j,k}(s)$. That is, we'll assume the swap is true and then show the result is equal to the pre-swapped form. This approach again helps to highlight the function of the tri-zeta term.

Assuming the swap, we can apply the inverse Laplace transform to the function $s^{-kz-1}$ in the inner integral, noting the following relation (see, for example, \cite{G_R}, Table 17.13),
\[
\frac{\Gamma(kz+1)}{z 2\pi i} \int_{s=a-i\infty}^{a+i\infty} \frac{e^{sw}}{s} s^{-kz} ds = \frac{w^{kz}}{z}
\]
to give for the post-swapped form of $(\ref{I_pre_swap})$,
\begin{equation} \label{I_jk_post_swap}
I_{j,k}^{\verb"~"} = \frac{1}{2\pi i} \int_{z=x-i\infty}^{x+i\infty} \frac{\zeta(kz+1)\zeta(jz)}{\zeta(jkz)}\frac{{ w}^{kz}}{z} dz.
\end{equation}

Applying the single-index sum form of the tri-zeta term from $(\ref{tri_zeta_positive})$ produces 
\[
I_{j,k}^{\verb"~"} = \frac{1}{2\pi i} \int_{z=x-i\infty}^{x+i\infty} \sum_{l=1}^{\infty} a_l l^{-z} \frac{{w}^{kz}}{z} dz.
\]
It is straightfoward to swap this integral and sum, as for all $l>w^k$, the integral vanishes (see for example \cite{Edwards}, section 3.5, for a very similar case). We have
\[
I_{j,k}^{\verb"~"} = \sum_{l=1}^{\infty} a_l \frac{1}{2\pi i} \int_{z=x-i\infty}^{x+i\infty} \left( \frac{w^k}{l}\right) ^z \frac{1}{z} dz
\]
which, by noticing each integral is just a step function with step of height $1/m$ at $w=mn_i^{j/k}$, shows by $(\ref{I_steps})$ that $I_{j,k}^{\verb"~"} = I_{j,k}$, making the fomula in $(\ref{I_jk_post_swap})$ valid for $I_{j,k}$.

\subsection{\bf The convolution exponential, and an overall form for $\mathbf{S_{j,k}(w)}$.}

Differentiating the $S_{j,k}$ step function of $(\ref{S_integral_1})$ with respect to $w$, we get the distributional derivative as a sum of delta distributions at points $v\in\mathbb{M}_{j,k}$, as well as at $v=0$. Taking the Laplace transform in the sense of generalized functions, and integrating over the series of delta distributions produces,

\begin{equation} \label{S_as_Laplace}
S_{j,k}(w) = \int_{0^-}^w \mathcal{L}^{-1} \left( e^{\ln Z(s)} \right)[t]dt - 1.
\end{equation}
Here we note that throughout the paper the $S_{j,k}$ integral's lower limit will be chosen as either $t=0^-$ or $t=0^+$, depending on the context. The difference amounts to whether the point mass at $w=0$ ($\delta_0(w)$) is captured. If it is, as in $(\ref{S_as_Laplace})$, since $0$ is not in $\mathbb{M}_{j,k}$, a $-1$ corrective term is required.

Next note that by $(\ref{I_function_ln_Z})$, $I_{j,k}$ can be written as,
\begin{equation} \label{I_as_Laplace}
I_{j,k}(w) = \mathcal{L}^{-1} \left( \frac{\ln Z_{j,k}(s)}{s} \right),
\end{equation}

and let $dI$ be its distributional derivative,
\[
dI_{j,k} = \mathcal{L}^{-1} \left( \ln Z_{j,k}(s) \right).
\]

Applying the convolution theorem for Laplace transforms, $\mathcal{L}^{-1}(\mathcal{L}f\mathcal{L}g) = f*g$, we have $\mathcal{L}^{-1}((\ln Z)^m) = dI^{*m}$, and we may now write the integrand of $(\ref{S_as_Laplace})$ as a convolution exponential with argument $dI$:
\begin{equation} \label{convo_exp_dI}
\mathcal{L}^{-1} \left( \mbox{exp}{(\ln Z}) \right) = \exp^*(dI),
\end{equation}
where by the definition of the convolution exponential,
\[
\exp^*(dI) = \delta_0+dI+\frac{dI*dI}{2!}+... = \mathcal{L}^{-1} \left( 1+\ln Z+\frac{(\ln Z)^2}{2!}+...\right),
\]
where we've used the shorthand $\delta_c(w) = \delta(w-c)$. Note that $dI$ may be regarded as a discrete measure with support only on $[1,\infty)$. The support of the $k$th convolution power of $dI$, $dI^{*k}$, will then be limited to $[k,\infty)$. This means that for any $w \in \mathbb{R}^+$, the exponential series of $(\ref{convo_exp_dI})$ as a function of $w$ is composed only of finitely many point masses up to and including $\lfloor w \rfloor$. This ensures $(\ref{convo_exp_dI})$ converges and is well-defined, as a counting measure.

The final form for $S_{j,k}$ is,
\begin{equation} \label{S_as_conv_exp}
S_{j,k}(w) = \int_{t=0^+}^{w} \exp^*(dI_{j,k}) dt = \int_{t=0^-}^{w} \exp^*(dI_{j,k}) dt - 1.
\end{equation}

\begin{remark}\label{zeta_parallel_remark}
Note that in parallel with $(\ref{S_integral_1})$ and $(\ref{S_as_conv_exp})$, with $\zeta(s)$ denoting the Riemann zeta function, and $N(w)$ the natural number counting function, a similar though less convenient formula exists relating the two,
\[
N(w) = \frac{1}{2\pi i} \int_{s=a-i\infty}^{a+i\infty} w^s \frac{e^{\ln \zeta(s)}}{s} ds = \int_{t=1^-}^{w} \exp^*(e^v dJ(e^v))[\ln t] \frac{dt}{t},
\]
where $J$ is defined as the step function with jumps of $1/n$ at each $p^n$, $p$ prime (see \cite{Edwards}, section 1.11), and $e^v dJ(e^v)$ produces point masses of weight $1/n$ at each $v=\ln (p^n)$.
\end{remark}

The convolution exponential formula, $(\ref{S_as_conv_exp})$, combined with the formulation $(\ref{I_pre_swap})$ for $I_{j,k}$, now produces an expression for $S_{j,k}$ in terms of the Riemann zeta function. The remainder of the paper involves estimates on $I_{j,k}$.

\section{\bf Developing $\mathbf{I_{j,k}(w)}$}

\subsection{\bf Overview of methods.}

To develop an estimate on $I_{j,k}(w)$, we have two formulas available. One is as a finite sum of k-free counting functions ($Q_k$'s), and the other is the contour integral of $(\ref{I_jk_post_swap})$. Because the contour integral form offers a more comprehensive approach, it will be used in the main estimates for $I_{j,k}(w)$, though the sum-of-$Q_k$'s version will also be used for a general order bound.

\subsubsection*{\bf The function $\mathbf{I_{j,k}(w)}$ as a sum of $\mathbf{Q_k}$'s}

A quick estimate of $I_{j,k}(w)$ can be derived by continuing the parallel between the multiplicative primes' step function $J(w)$ mentioned in Remark $\ref{zeta_parallel_remark}$, and the parallel step function, $I(w)$. That is, just as $J(x)$ is related to $\pi(x)$, the prime counting function, by,
\[
J(x) = \pi(x) + \frac{1}{2} \pi(x^{1/2}) + \frac{1}{3} \pi(x^{1/3}) + ...,
\]
the following relation between $I_{j,k}(w)$ and $Q_k(w)$ (the k-free counting function) is easy to verify,

\begin{equation} \label{I_jk_as_Q_sums}
I_{j,k}(w) = Q_k \left( \left( w \right)^{k/j} \right) + \frac{1}{2}Q_k \left( \left( \frac{w}{2} \right)^{k/j} \right)  + ... = \sum_m \frac{1}{m} Q_k\left( \left( \frac{w}{m} \right)^{k/j} \right),
\end{equation}
where the series terminates when $m>w$.

\begin{remark}\label{Q_k_remark}
A reader only interested in first order approximations to $S_{j,k}(w)$ may at this point skip ahead to section \ref{sum_of_Qs_section}, where a first order approximation to $I_{j,k}(w)$ is derived using just the $Q_k$ series in $(\ref{I_jk_as_Q_sums})$. Any of the subsequent first order approximations derived for $S_{j,k}(w)$ at that point will be valid, based on the developments to this point.
\end{remark}

\subsubsection*{\bf The function $\mathbf{I_{j,k}(w)}$ as a sum of residues}

Examining the integrand of $(\ref{I_jk_post_swap})$, 
\[
\frac{\zeta(kz+1)\zeta(jz)}{\zeta(jkz)}\frac{{w}^{kz}}{z},
\]
let $\rho$ as usual represent the non-trivial zeros of the zeta function. The integrand has poles at $z = 1/j$, $z = \rho/jk$,  $z = 0$, and $z=-2/(jk), -4/(jk), ...$ except for at $z=-2m/(jk)$ when $m=(2q+1)j/2$ or $m=qk$, $m$ and $q$ in $\mathbb{Z}^+$ (where the former equality can only apply when $j$ is even). The pole at the least negative position then is at $z=-2/(jk)$ and the strategy will be to displace the vertical contour of $(\ref{I_jk_post_swap})$ to just skirt the line $x=-2/(jk)$. The Cauchy residue theorem can then be applied to the region between the old and new contours, while an order bound estimate will be found for the new, displaced contour itself (the `remainder integral'). Two fundamentally different cases arise in bounding the remainder integral. When $j=1$ and $k=2$, the remainder integral does not converge absolutely and the most work is needed to find an order bound. For $j=1$, $k \ge 3$, and for $j\ge 2$, $k\ge2$, the remainder integral converges absolutely and an order bound follows easily.

The next sections address expressing $I_{j,k}(w)$ as a sum of residues, with remainder.

\subsection{\bf Residue form for $\mathbf{I_{j,k}(w)}$}

The first step will be to displace the vertical contour of $(\ref{I_jk_post_swap})$ to just skirt the first negative pole, $x=-2/(jk)$, in the $+\mbox{Re}$ direction. Later sections will address application of the Cauchy residue theorem.

To understand the obstacles involved in displacing the vertical contour to $x<0$, we'll have a look at the magnitude of the integrand of $I_{j,k}(w)$'s contour integral, $(\ref{I_jk_post_swap})$, for large $|y|$. To accomplish this, the following reformulation of the integrand of $(\ref{I_jk_post_swap})$ is useful. It follows from the reflection formula for the Riemann zeta function, $\zeta(z) = 2^z \pi^{z-1} \sin(\pi z / 2) \Gamma(1-z) \zeta(1-z)$:
\begin{equation} \label{I_jk_flipped}
I_{j,k}(w) = \frac{1}{2\pi i} \int_{z=x-i\infty}^{x+i\infty} \frac{2}{k}(2\pi)^{(k+j-jk)z} \frac{\cos(kz\pi/2)\sin(jz\pi/2)}{\sin(jkz\pi/2)} \frac{\Gamma(-kz) \Gamma(-jz)}{\Gamma(-jkz)} \frac{\zeta(-kz) \zeta(1-jz)}{\zeta(1-jkz)} \frac{{ w}^{kz}}{z} dz.
\end{equation}

This gives a magnitude for the integrand, after making large $|y|$ approximations to all but the tri-zeta term, of
\begin{equation} \label{I_jk_flipped_large_y}
\approx c_0 |y|^{(jk-j-k)x} |y|^{-3/2} \left|  \frac{\zeta(-kz) \zeta(1-jz)}{\zeta(1-jkz)} \right| w^{kx},
\end{equation}
where $c_0$ is constant for fixed $x$.

The main obstacle to displacing the vertical contour is in bounding the new, `post-reflected' tri-zeta term. The following lemmas are useful:
\begin{lemma}\label{zeta_reciprocal}
The reciprocal of the Riemann zeta function,
\[
\frac{1}{\zeta(z)},
\]
is bounded by a constant along a given vertical contour with $x > 1$ and uniformly in the half-plane $x \ge 1+\delta$, $\delta>0$.
\end{lemma}
\begin{proof}
This follows immediately from the relation,
\[
\frac{1}{\zeta(z)} = \sum_n \frac{\mu(n)}{n^z},
\]
where $\mu(n)$ is the Mobius function, since,
\[
\left| \sum_n \frac{\mu(n)}{n^z}\right| \le \sum_n \frac{| \mu(n) |}{n^x} < \sum_n \frac{1}{n^x} = \zeta(x) < \infty,
\]
whenever $x>1$. The uniformity condition follows from $| \zeta(z) | \le \zeta(1+\delta) $ whenever $x \ge 1+\delta$.
\end{proof}

\begin{lemma} \label{zeta_Lindelhof_bounds}
(see \cite{Edwards}, Section 9.2) The zeta function may be bounded as follows (Lindel\"of),
\begin{equation} \label{zeta_bound_n1}
| \zeta(z) | = \mathcal{O} (\ln|y|) \mbox{ whenever $x\ge 1$, for all $|y|\ge 2$},
\end{equation}
\begin{equation} \label{zeta_bound_n2}
| \zeta(z) | = \mathcal{O} (|y|^{1/2 - x/2 + \epsilon}) \mbox{ whenever $0\le x \le 1$, for all $|y|\ge 1$, $\epsilon>0$},
\end{equation}
\begin{equation} \label{zeta_bound_n3}
| \zeta(z) | = \mathcal{O} (|y|^{1/2 - x + \epsilon}) \mbox{ whenever $x\le 0$, for all $|y| > y_1$ , $\epsilon>0$}.
\end{equation}
\end{lemma}

Applying the lemmas to the tri-zeta term of $(\ref{I_jk_flipped_large_y})$ we have, when $x\le -\delta < 0$,
\[
|\zeta(1-jz)| = \mathcal{O}(1)
\]
and by Lemma $\ref{zeta_reciprocal}$, 
\[
\left| \frac{1}{\zeta(1-jkz)} \right| = \mathcal{O}(1).
\]
Also, by Lemma $\ref{zeta_Lindelhof_bounds}$'s equations $(\ref{zeta_bound_n1})$ and $(\ref{zeta_bound_n2})$, we have, 
\[
|\zeta(-kz)| =
  \left\{
    \begin{array}{ll}
	\mathcal{O} (|y|^{1/2 - k|x|/2 + \epsilon}) & \mbox{ when $-1/k < x \le -\delta$}, \\
	\mathcal{O} (\ln |y|) & \mbox{ when $x \le -1/k$}.
    \end{array}
  \right.
\]
Using $(\ref{I_jk_flipped_large_y})$ and the results just above, the overall bound on the integrand of $(\ref{I_jk_post_swap})$ when $|y|\ge2$ becomes
\begin{equation} \label{disp_bound_low}
\left| \frac{\zeta(kz+1)\zeta(jz)}{\zeta(jkz)} \frac{{ w}^{kz}}{z} \right| =
  \left\{
    \begin{array}{ll}
	\mathcal{O} \left(|y|^{(jk-j-k/2)x -1 + \epsilon} \right) & \mbox{ when $-1/k < x \le -\delta$}, \\
	\mathcal{O} \left(|y|^{(jk-j-k)x - 3/2 + \epsilon} \right) & \mbox{ when $x \le -1/k$}.
    \end{array}
  \right.
\end{equation}

To examine various ranges of $j$ and $k$, recalling the remarks following $(\ref{S_integral_1})$, that $j/k$ be in lowest form:

\begin{description}
\item[$\bf{j=1}$]
When $j=1$, the bound reduces to 
\begin{equation} \label{I_remainder_abv_1_k}
\mathcal{O} \left(|y|^{(k/2-1)x -1 + \epsilon} \right) \mbox{ when $-1/k < x \le -\delta$}
\end{equation}
and 
\begin{equation} \label{I_remainder_bel_1_k}
\mathcal{O} \left(|y|^{-x -3/2 + \epsilon} \right) \mbox{ when $x \le -1/k$}.
\end{equation}

Notice
\begin{enumerate}
\item By $(\ref{I_remainder_bel_1_k})$, vertical contours have no chance of convergence if $x\le -3/2$.
\item Also by $(\ref{I_remainder_bel_1_k})$, vertical contours won't converge absolutely if $x \le -1/2$. And if additionally $k=2$, then by equation $(\ref{I_remainder_abv_1_k})$, vertical contours won't converge absolutely if $x\in (-1/2,-\delta]$ either, making the case $j=1$, $k=2$ unique since vertical contours won't converge absolutely whenever $x<0$.
\item If $k\ge3$, equation $(\ref{I_remainder_bel_1_k})$ shows the vertical contour will converge absolutely whenever $-1/2+\epsilon < x \le -1/k$. Specifically, notice $x=-c_c/(jk)$, where $c_c = 3/2 - \epsilon_0$, for some suitably chosen $\epsilon_0$, lies in this region of absolute convergence.
\end{enumerate}

\item[$\bf{j\ge2}$]
In general, the bound to ensure absolute convergence when in $-1/k<x<-\delta$ is
\begin{equation} \label{I_a_remainder_abv_1_k}
x < \frac{-\epsilon}{jk-j-k/2},
\end{equation}
and notice that since $jk-j-k \ge 1$ when $j\ge 2$ ($k \ge 2$), the bound when $x \le -1/k$ is,
\begin{equation} \label{I_a_remainder_bel_1_k}
\mathcal{O} \left( |y|^{x -3/2 + \epsilon} \right).
\end{equation}
So by $(\ref{I_a_remainder_abv_1_k})$, we again have absolute convergence along the contour $x=-c_c/(jk)$, $c_c$ as above.

\end{description}

\subsubsection*{\bf Displacing the vertical contour.}

To develop the displaced contour in all cases of $j,k$, let $C_r$ be a contour with $x$ fixed in $1/j<x$ and $|y|<h$, let $C_l$ be the contour with $x=-c_c/(jk)$ and $|y|<h$, and let contours $C_t$ and $C_b$ be the segments at $y=h$ and $y=-h$ that complete the rectangle initiated by $C_r$, $C_l$. See Figure~\ref{fig:vertical_displace_contours}.

\begin{figure}[h]
\centerline{\includegraphics{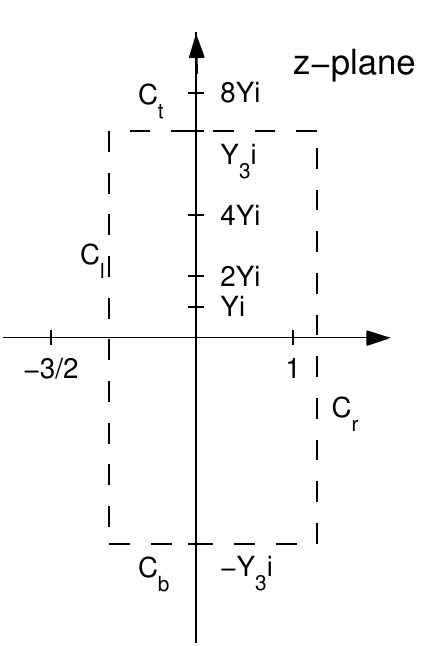}}
\caption{\emph{Displacing the $I_{j,k}(w)$ vertical contour.}}
\label{fig:vertical_displace_contours}
\end{figure}

To justify shifting the vertical contour, we need to show that the integrand of $(\ref{I_jk_post_swap})$ along contours $C_t$ and $C_b$ goes to zero as $h\rightarrow \infty$. The largest issue is handling the denominator of the tri-zeta term. For example, one of the problems with bounding the reciprocal of the zeta function along horizontal contours through the critical strip is the risk of passing arbitrarily close to a zeta zero. However, there exist certain safe lines of passage through this region as the following Lemma guarantees.

\begin{lemma} \label{zeta_reciprocal_Bartz}
(See \cite{Bartz}, Lemma 1). There exist positive constants $\lambda_1$, $\lambda_2$, and $y_0$ such that for any $Y \ge y_0$, we can find a $y$ between $Y$ and $2Y$ so that,
\[
\left| \frac{1}{\zeta(x+iy)} \right| \le \lambda_2 \ln^{\lambda_1}y \mbox{ for $-1\le x \le 3$}.
\]
\end{lemma}

To allow use of the lemma, define the sequence $\{ Y_m \}$ as the sequence of positive real values guaranteed by the lemma where we will choose $Y > \max \{y_0, y_1, 2 \}$ (with $y_0$ as required by the lemma, and, as we'll see, the values $y_1$ and $2$ as required by Lemma $\ref{zeta_Lindelhof_bounds}$) and each $Y_m \in (2^{m-1} Y,2^{m} Y)$ so that the Lemma's bound on $1/ \zeta(jkz)$ holds for all $Y_m$, independent of $m$.

We'll split the horizontal contour $C_t$ (the contour $C_b$ is treated similarly), into three sections and show how to make each section's contribution disappear as $|y|$ grows large. The sections are: $-c_c/(jk) < x \le -\delta$ , $-\delta < x < 1/(jk) + \delta$, and $1/(jk)+\delta \le x < 1/j + \delta$ (where, by definition, the original vertical contour integral of $(\ref{I_jk_post_swap})$ was independent of the value of $x$, provided $x>1/j$).

\medskip
\noindent
\emph{The section $-\delta < x < 1/(jk) + \delta$.}

This `middle' section requires the most care and will be dealt with first. To keep the $1/\zeta(jkz)$ term in the integrand of $(\ref{I_jk_post_swap})$ under control, we'll invoke Lemma $\ref{zeta_reciprocal_Bartz}$. To do so, we position the contours $C_t$ and $C_b$ at $\pm (Y_m/(jk)) i$ so the lemma's bounds apply. We'll also use the bounds for the zeta function from equations $(\ref{zeta_bound_n1})-(\ref{zeta_bound_n3})$.

For the $|1/\zeta(jkz)|$ term, Lemma $\ref{zeta_reciprocal_Bartz}$ produces
\begin{equation} \label{zeta_jkz_recip_bound}
\left| \frac{1}{\zeta(jkz)} \right| < c_0 ln^{c_1} (Y_m) \mbox{ for $-1/(jk) \le x \le 3/(jk)$}.
\end{equation}
(Notice the lower limit for $x$ in $(\ref{zeta_jkz_recip_bound})$ is $-1/(jk)$, requiring, at least, $\delta \le 1/(jk)$.)

For $|\zeta(kz+1)|$, the bound of equation $(\ref{zeta_bound_n2})$ gives
\[
|\zeta(kz+1)| = \mathcal{O} (Y_m^{1/2-(1+k\delta)/2+ \epsilon}),
\]
provided $\delta < 1/k$, and the bound $(\ref{zeta_bound_n3})$ gives
\[
|\zeta(jz)| = \mathcal{O}( Y_m^{1/2 + \delta j}).
\]

Overall then,
\[
\left| \frac{\zeta(kz+1)\zeta(jz)}{\zeta(jkz)} \frac{{ w}^{kz}}{z} \right| = \mathcal{O} \left( ln^{c_1} (Y_m) Y_m^{-1/2+\delta(k/2+j) +\epsilon} \right),
\]
and it suffices to make
\[
\delta < \frac{1}{3k+6j},
\]
which produces an overall bound on the integrand of
\begin{equation} \label{disp_bound_mid}
\left| \frac{\zeta(kz+1)\zeta(jz)}{\zeta(jkz)} \frac{{ w}^{kz}}{z} \right| =\mathcal{O}\left(Y_m^{-1/3+\epsilon_1}\right)
\end{equation}
provided (combining all restrictions on $\delta$),
\[
\delta < \min \left\{ \frac{1}{k}, \frac{1}{jk}, \frac{1}{3k+6j} \right\} = \min \left\{ \frac{1}{jk}, \frac{1}{3k+6j} \right\}.
\]

\medskip
\noindent
\emph{The section $-c_c/(jk) \le x \le -\delta $.} 

This section follows from equations $(\ref{disp_bound_low})$ and the worst case scenario at $j=1$, $k=2$, for an overall bound of

\[
\left| \frac{\zeta(kz+1)\zeta(jz)}{\zeta(jkz)} \frac{{ w}^{kz}}{z} \right| =\mathcal{O}\left(Y_m^{-3/4+\epsilon}\right)
\]

\medskip
\noindent
\emph{The section $1/(jk) + \delta \le x < 1/j + \delta $.}

Since in this section
\[
|\zeta(jz)| = \mathcal{O} (Y_m^{1/2+\epsilon}),
\]
by Lemma $\ref{zeta_Lindelhof_bounds}$, we have the bound,
\begin{equation} \label{disp_bound_high}
\left| \frac{\zeta(kz+1)\zeta(jz)}{\zeta(jkz)} \frac{{ w}^{kz}}{z} \right| = \mathcal{O} (Y_m^{-1/2+\epsilon}).
\end{equation}

\medskip
Combining the bounds $(\ref{disp_bound_mid})$-$(\ref{disp_bound_high})$ on the integrand of $(\ref{I_jk_post_swap})$ shows that the integral along finite contours $C_t$ (and $C_b$) vanishes as $Y_m \rightarrow \infty$, justifying the displacement of the vertical contour of $(\ref{I_jk_post_swap})$ to $x=-c_c/(jk)$, \emph{provided} the displaced integral itself converges. That is, it has been shown,
\begin{equation} \label{I_displaced_contour}
\frac{1}{2\pi i} \int_{z=x_h-i\infty}^{x_h+i\infty} \frac{\zeta(kz+1)\zeta(jz)}{\zeta(jkz)}\frac{{ w}^{kz}}{z} dz = \sum \mbox{Res }f(z) +  \frac{1}{2\pi i} \int_{z=x_l-i\infty}^{x_l+i\infty} \frac{\zeta(kz+1)\zeta(jz)}{\zeta(jkz)}\frac{{ w}^{kz}}{z} dz,
\end{equation}
provided at least the improper integral on the right hand side converges, and where $x_l = -c_c/(jk)$ (just skirting the $+\mbox{Re}$ side of the pole at $-2/(jk)$), $1/j<x_h$, and $f(z)$ represents the integrand of $(\ref{I_jk_post_swap})$.

\subsubsection*{\bf Residue form of $\mathbf{I_{j,k}(w)}$}

The residues in equation $(\ref{I_displaced_contour})$ are easily computed:
\begin{equation} \begin{split} \label{I_residue}
I_{j,k}(w) = \frac{\zeta(1+k/j)}{\zeta(k)} w^{k/j}& + \lim_{m\to\infty}\stackrel{*}{\sum_{|\mbox{Im }\rho|<Y_m/(jk)}} \frac{\zeta(\rho/j+1)\zeta(\rho/k)} {\zeta ^\prime (\rho)}\frac{{ w}^{\rho/j}}{\rho} + \ln w + \gamma - j\left( 1- \frac{1}{k} \right) \ln 2\pi \\
&+ \frac{1}{2\pi i} \int_{z=-c_c/(jk)-i\infty}^{-c_c/(jk)+i\infty} \frac{\zeta(kz+1)\zeta(jz)}{\zeta(jkz)}\frac{{ w}^{kz}}{z} dz,
\end{split} \end{equation}
where $\gamma$ is the Euler-Mascheroni constant. Two remarks are in order about the sum over zeta zeros:
\begin{enumerate}
\item The $Y_m$ sequence for the sum \emph{is} needed--it guarantees (by Lemma $\ref{zeta_reciprocal_Bartz}$) the limit implicit in the sum will converge to the proper value. 
\item The summands in the zeta zero sum only represent those for \emph{simple} zeros. The full expression for a zeta zero of multiplicity $l_{\rho}$ is:
\[
\frac{1}{(l_{\rho}-1)!} \frac{d^{l_{\rho}-1}}{dz^{l_{\rho}-1}} \left. \left( (z-\rho/jk)^{l_{\rho}} \frac{\zeta(kz+1)\zeta(jz)}{\zeta(jkz)}\frac{{ w}^{kz}}{z} \right) \right|_{z=\rho/jk}.
\]
It is not known whether all the zeta zeros are simple, though the first $1.5e9$ are (see \cite{VanDeLune}). The $*$ above the $\Sigma$ summation symbol will indicate this simplification.

\end{enumerate}

For equation $(\ref{I_residue})$ to be valid, it remains to show the remainder integral converges which we'll do next.

\subsection{\bf Evaluating the remainder integral, in general.}

To ensure convergence of the remainder integral in $(\ref{I_residue})$, namely,
\begin{equation} \label{I_remainder}
\frac{1}{2\pi i} \int_{z=x-i\infty}^{x+i\infty} \frac{\zeta(kz+1)\zeta(jz)}{\zeta(jkz)}\frac{{ w}^{kz}}{z} dz,
\end{equation}
we'll evaluate it along $x=-c_c/(jk)$, recalling $c_c = 3/2-\epsilon_0$. There are two distinct cases. By equations $(\ref{I_remainder_abv_1_k})$ - $(\ref{I_a_remainder_bel_1_k})$ and the surrounding discussion, we know $(\ref{I_remainder})$ converges absolutely along $x=-c_c/(jk)$ for all allowed combinations of $j$ and $k$ \emph{except} the instance $j=1$, $k=2$. So the cases are:

\begin{description}
\item[$\bf{j\ge1, k\ge3}$]
In this case, the integral $(\ref{I_remainder})$ converges absolutely along $x=-c_c/(jk)$ and therefore has order bound,
\begin{equation} \label{non_1_2_order_bound}
\frac{1}{2\pi i} \int_{z=x-i\infty}^{x+i\infty} \frac{\zeta(kz+1)\zeta(jz)}{\zeta(jkz)}\frac{{ w}^{kz}}{z} dz = \mathcal{O}\left( w^{-c_c/j} \right).
\end{equation}
Note this order bound is not necessarily uniform over $j$ and $k$. Also note the order bound is about the best we can do since it can't be improved by extending to a contour with $x<-2/(jk)$ because of the pole at $x=-2/(jk)$.

\item[$\bf{j=1, k=2}$]
This case requires considerably more work. The approach will be to use the post-reflected form of the integrand from equation $(\ref{I_jk_flipped})$,
\begin{equation} \label{tri_zeta_neg_1_2}
\frac{\zeta(-2z) \zeta(1-z)}{\zeta(1-2z)} \frac{1}{2z} \frac{\cos(z\pi)}{\cos(\frac{z\pi}{2})} \Gamma(-z)  (\sqrt{2\pi} w)^{2z},
\end{equation}
and expand the tri-zeta term as a sum then evaluate the resulting sum of integrals using steepest descent. The next sections develop this approach.
\end{description}

\subsection{\bf Evaluating the remainder integral, for the case $\mathbf{j=1, k=2}$.}
The approach is by divide-and-conquer, expressing the tri-zeta term of $(\ref{tri_zeta_neg_1_2})$ as a series, then evaluating this series of integrals, and finally showing the integral of the series, and the series of integrals are equivalent.

To develop an expansion for the negative tri-zeta term of equation $(\ref{tri_zeta_neg_1_2})$, we'll use the relation (see \cite{Wolfram_1}),
\[
\frac{\zeta(z-1)}{\zeta(z)} = \sum_{n=1}^{\infty} \frac{\phi(n)}{n^z}
\]
which holds when $\mbox{Re }z = x > 2$ and where $\phi$ is the Euler totient function.

This gives 
\[
\frac{\zeta(-2z)\zeta(1-z)}{\zeta(1-2z)} = \zeta(1-z) \sum_{n=1}^{\infty} \frac{\phi(n)}{n^{1-2z}}
\]
whenever $x<-\frac{1}{2}$. Expanding the $\zeta(1-z)$ term in its Dirichlet series gives the double sum of two absolutely convergent series:
\[
\sum_{n,k} \frac{\phi(n)}{k^{1-z}n^{1-2z}} = \sum_{n,k} \frac{n\phi(n)}{(kn^2)^{1-z}},
\]
where the reformulation on the right hand side has been made in the aim of expressing this as a sum in a single index. As discussed, any $j$ in $\mathbb{Z}^+$ can be factored uniquely as $a b^2$, where $a$ is a square-free number. Let $\mu_2(k)$ as above be the square-free indicator function. To consolidate terms in the double sum, consider the coefficient of the term $\mu_2(k)/(kn^2)^{1-z}$. For $j=kn^2$, $k$ square-free, if we remove the restriction on $k$ being square-free, note $j$ can be written as $k_1 n_1^2$ once for every divisor of $n$ (for example, if $j=6\cdot10^2$, $6$ square-free, it can also be written as $24\cdot5^2$, $150\cdot2^2$, and $600\cdot1^2$). This allows reformulating the double sum as,
\[
\sum_{n,k} \frac{\mu_2(k) \sum_{d|n} d\phi(d)}{(k n^2)^{1-z}}.
\]

So we have,
\begin{equation} \label{hmj_v_zeta}
\frac{\zeta(-2z)\zeta(1-z)}{\zeta(1-2z)} = \sum_{l=1}^{\infty} \frac{ h_l }{ l^{1-z} }, \mbox{ provided $\mbox{Re }z < -1/2$},
\end{equation}
where $l$ has been factored uniquely as $l=n m^2$, $n$ square free, and where 
\[
h_l = \sum_{d|m} d\phi(d).
\]
The sum is guaranteed to converge as the rearrangement of terms in the product of two absolutely convergent sums.

Now noting that $(\ref{I_jk_flipped})$ simplifies with $j=1$ and $k=2$, and applying the infinite sum reformulation of the tri-zeta of $(\ref{hmj_v_zeta})$, we'll exchange the order of summation and integration in $(\ref{I_jk_flipped})$ to get,
\begin{equation} \label{I_sum_over_j}
\sum_{l=1}^{\infty} \frac{1}{2 \pi i} \int_{z=x-i\infty}^{x+i\infty}  \frac{ h_l }{ l^{1-z} } \frac{1}{2z} \frac{\cos(z\pi)}{\cos(\frac{z\pi}{2})} \Gamma(-z) (\sqrt{2\pi} w)^{2z} dz,
\end{equation}
with the goal of finding a bound on it and showing the order of integration and summation can be swapped to prove its equivalence to $(\ref{I_remainder})$.

To evaluate the $l^{th}$ summand of $(\ref{I_sum_over_j})$, choose $M\ge\frac{3}{2}$, and split the contour into $C_1 \equiv \{ x+iy \mid | y | \le 2M \}$, $C_2 \equiv \{ x+iy \mid y \ge 2M \}$ and  $C_3 \equiv \{ x+iy \mid y \le -2M \}$.

The contour $C_1$ has a straightforward absolute bound, 
\[
\left| \frac{1}{2 \pi i} \int_{C_1} \frac{ h_l }{ l^{1-z} } \frac{1}{2z} \frac{\cos(z\pi)}{\cos(\frac{z\pi}{2})} \Gamma(-z) (\sqrt{2\pi} w)^{2z} dz \right| <
\]
\begin{equation} \label{middle_portion_hmj}
\frac{1}{2 \pi} \int_{C_1} \left| \frac{ h_l }{ l^{1-z} } \frac{1}{2z} \frac{\cos(z\pi)}{\cos(\frac{z\pi}{2})} \Gamma(-z) (\sqrt{2\pi} w)^{2z} \right| dz = c_0 \frac{ h_l }{ l^{1-x} w^{-2x}} 
\end{equation}

For contours $C_2$ and $C_3$ we'll expand the integrand in this region asymptotically. Generalize the lower limit to $2M \le |y_0|$ to give, in the case of $C_2$,
\begin{equation} \label{hmj_simpler_integrand}
\frac{ h_l }{ l } \frac{1}{4 \pi i} \int_{x+y_0i}^{x+i\infty} \frac{1}{z} \frac{\cos(z\pi)}{\cos(\frac{z\pi}{2})} \Gamma(-z) (2\pi l w^2)^z dz,
\end{equation}
and similarly for $C_3$. 

The asymptotic expansion for the gamma function (see, eg, \cite{Abramowitz}) and an expansion of the cosine fraction in terms of its component exponentials gives for those terms,
\[
\Gamma(-z) \frac{\cos(z\pi)}{\cos(\frac{z\pi}{2})} = \sqrt{2 \pi} i e^z e^{i\frac{\pi}{2}z} z^{-(z+\frac{1}{2})} R_0 \left( 1+ R_1(z) + R_2(z) \right) \mbox{   when $y \ge \delta > 0$}.
\] 
Assuming $\mbox{Im }z \ge M$, $R_0$ is the correction factor from the gamma approximation and may be bounded by $e^{c_1/M}$, while $R_1$ and $R_2$ arise from the cosine fraction expansion and may be bounded by $e^{-c_2/M}$.

We can now rewrite $(\ref{hmj_simpler_integrand})$ as, 
\begin{equation} \label{asymptote_integrand}
\frac{ h_l }{ l } \frac{1}{2\sqrt{2 \pi}} \int_{x+y_0i}^{x+i\infty} e^{i\frac{\pi}{2}z} z^{-(z+\frac{3}{2})} (2\pi l w^2 e)^z R_0(1+R_1(z)+ R_2(z)) dz.
\end{equation}

\subsection{\bf Steepest descent in general for the $\mathbf{j=1}$, $\mathbf{k=2}$ case.}

We now evaluate the integral of equation $(\ref{asymptote_integrand})$, ignoring for the moment its remainder terms:
\begin{equation} \label{jth_summand_asymptote}
\frac{ h_l }{ l } \frac{1}{2\sqrt{2 \pi}} \int_{x+iy_0}^{x+i\infty} e^{i\frac{\pi}{2}z} z^{-(z+\frac{3}{2})} (2\pi l w^2 e)^z dz .
\end{equation}

An estimate of this integral is straightforward by use of the method of steepest descent. The change of variables $s = \frac{z}{l w^2}$ in $(\ref{jth_summand_asymptote})$ produces 
\begin{equation} \label{s_form} 
\frac{ h_l }{ l^{\frac{3}{2}} w} \frac{1}{2\sqrt{2 \pi}} \int_{s=\frac{c}{l w^2}+\frac{iy_0}{lw^2}}^{\frac{c}{l w^2}+i\infty} s^{-\frac{3}{2}} \exp\left[ l w^2 s \left( \frac{i\pi}{2} + 1 + \ln \left(\frac{2\pi}{s} \right) \right) \right] ds
\end{equation}
where the integral is now of saddle point form 
\begin{equation} \label{standard_SD}
K(\lambda) = \int_C g(s) e^{\lambda f(s)} ds.
\end{equation}

The saddle points of $(\ref{s_form})$ are found by solving $f^{\prime}(s)=0$ to give a unique saddle point, $s_0 = 2\pi i$; by using known evaluation formulas for $(\ref{standard_SD})$ when $|\lambda|= l w^2$ is large (see, for example, \cite{Hassani}), equation $(\ref{s_form})$ evaluates to
\begin{equation} \label{SP_estimate}
\approx -\frac{h_l}{l^2} \frac{1}{4\pi w^2} e^{i 2\pi l w^2 }
\end{equation}
to a first approximation whenever $0 \le y_0 << 2\pi l w^2$ (since the contour can be deformed to pass through the saddle point without modification because the integrand has no intervening singularities).

Combining $(\ref{SP_estimate})$ with the expectation that $(\ref{s_form})$ is $\approx 0$ when $2\pi l w^2 << y_0$, suggests that the individual integrals as well as the overall sum of $(\ref{I_sum_over_j})$ at least converge, and that with some additional requirements on the rate of convergence in $z$ of the integrals, that exchanging the order of integration and summation is justified. The next section will show that this is the case.

\subsection{\bf Steepest descent in particular for the $\mathbf{j=1}$, $\mathbf{k=2}$ case.}

To find precise bounds on $(\ref{jth_summand_asymptote})$, consider its integrand:

\begin{equation} \label{jth_summand_absolute}
\left| \frac{ h_l }{ l } \frac{1}{2\sqrt{2 \pi}} e^{i\frac{\pi}{2}z} z^{-(z+\frac{3}{2})} (2\pi l w^2 e)^z \right| = \frac{ h_l }{ l } \frac{1}{2\sqrt{2 \pi}} r^{-3/2} e^{y(\theta - \pi/2)} \left( \frac{2\pi l w^2 e}{r} \right) ^x
\end{equation} 
where $z=x+iy=re^{i\theta}$, and the saddle point is $z=i \alpha_l = 2\pi i l w^2$. 

The integral $(\ref{jth_summand_asymptote})$ can be divided into three cases, depending on where the lower limit $x+i y_0$ lies with respect to the saddle point $i \alpha_l$. Bounds can be found for each case, by evaluating the integral of the magnitude of the integrand where knowledge of the gradient is used in deforming the contour to provide a suitable bound. The computations to find the bounds are not difficult but involve some algebra and have been put in the Appendix because of their length. Assuming $w$ is fixed and greater than zero and that $2M \le y_0$, the three cases and their bounds are as follows:

\begin{description}

\item[Case I] In this case, $\alpha_l e < y_0$. The bound is,
\begin{equation} \label{case_I_bound}
\frac{h_l}{2\sqrt{2\pi}l} \left( c^{|x|} e^{|x|} |x| (\alpha_l e)^{-3/2} + \frac{3}{2} y_0^{-3/2} \left( 1-e^{-(2/3)y_0} \right) + \frac{2}{\ln 2} y_0^{-3/2} e^{-y_0 (\pi/4+\ln(\sqrt{2}))}  \right)
\end{equation}
\[
= \mathcal{O}\left( \frac{h_l}{l^{5/2} w^3} \right),
\]
which follows from Case I's requirement that $2\pi l w^2 e < y_0$.
\item[Case II] In this case, $\alpha_l < y_0 \le \alpha_l e$. The bound is 
\begin{equation} \label{case_II_part_bound}
\frac{h_l}{2\sqrt{2\pi}l} \left( c^{|x|} e^{|x|} |x| (\alpha_l)^{-3/2} + \sqrt{\frac{\pi}{2}}y_0^{-1} +  \frac{3}{2} y_0^{-3/2} \left( 1-e^{-(2/3)y_0}  \right) + \frac{2}{\ln 2} y_0^{-3/2} e^{-y_0 (\pi/4+\ln(\sqrt{2}))} \right),
\end{equation}
The overall order bound for this Case is,
\[
=\mathcal{O}\left( \frac{h_l}{l^2 w^2} \right),
\]
which similarly follows because of Case II's requirement $2\pi l w^2 < y_0$. 

\item[Case III] In this case, $y_0 \le \alpha_l$. This bound is formed by summing three parts. For the first, we'll evaluate
\[
\frac{h_l}{2\sqrt{2\pi}l} \left( \left( \frac{\alpha_l e}{y_1 - x} \right)^{x} (y_1 + x)^{-3/2} e^{-x} |x| + \frac{\sqrt{\pi}}{2^{1/4}} y_1^{-1-x} {\alpha_l}^{x} \right)
\]
at $y_1 = y_0$, then add to the result the same expression evaluated at $y_1 = \alpha_l$ . For the second, we have,
\[
\frac{h_l}{2\sqrt{2\pi}l} \left( 2 e^{M \pi/2} \left( \frac{\alpha_l e}{\sqrt{2}M} \right)^{-M} \left( \frac{1}{\sqrt{y_0-M}} - \frac{1}{\sqrt{\alpha_l-M}}\right) \right).
\]
And for the third, we use the bounds from Case II, setting $y_0=\alpha_l$ in $(\ref{case_II_part_bound})$.
Given $M = 3 / 2$, the overall order bound for this case is
\[
\mathcal{O}\left( \frac{h_l}{l^{\mu} w^{\nu}} \right)
\]
where $\mu = 1-x$ and $\nu = -2x$ if $-1 < x < -1/2$ and $\mu = 2$, $\nu = 2$ otherwise.
\end{description}

The goal was to bound the asymptotic expansion of $(\ref{hmj_simpler_integrand})$ along the $|y|\ge2M$ contours $C_2$ and $C_3$. We have bounded it with the Case bounds, the bounds on the remainder terms (which at worst affect the Case bounds by a constant of scale), and the fact though the Case bounds were derived for the $+\mbox{Im}$ contour ($C_2$), the function is analytic, so the bounds need only be multiplied by a factor of at most $2$ to accomodate a bound on the sum of the symmetric contours.

Given the goal was to evaluate the remainder integral $(\ref{asymptote_integrand})$ along $x=-c_c/(jk) = -3/4 + \epsilon$ (recalling $c_c = 3/2-\epsilon_0$), it is time to prove the following theorem:

\begin{theorem} \label{I_remainder_thm}
The remainder integral, $(\ref{I_remainder})$, when $j=1$ and $k=2$ converges and obeys order bound,
\begin{equation} \label{I_remainder_w_order_bound}
\frac{1}{2 \pi i} \int_{z=x-i\infty}^{x+i\infty} \frac{\zeta(2z+1) \zeta(z)}{\zeta(2z)} \frac{{ w}^{2z}}{z} dz = \mathcal{O} \left( w^{-3/2+\epsilon} \right),
\end{equation}
for $x = -3/4 + \epsilon$.
\end{theorem}
\begin{proof}
The following basic proof rests on showing that the series of integrals in $(\ref{I_sum_over_j})$ converges uniformly in $y$ (see for example \cite{Edwards}, section 3.5, for an almost identical proof). Let $g_l$ denote the integrand of $(\ref{I_sum_over_j})$. First note that $\sum_l g_l$ converges uniformly for all $y$ along finite vertical contours. This follows from the initial integrand estimate $(\ref{hmj_simpler_integrand})$ and the subsequent absolute bounds on $g_l$, $(\ref{jth_summand_absolute})$. We in fact have the bound
\[
\left| g_l \right| < \frac{h_l}{l^{1-x}} c_0 r^{-3/2-x} e^{y(\theta-\pi/2)} < \frac{h_l}{l^{1-x}} c_1,
\]
for all $y>2M$, and where recall the series in $l$ converges provided $1-x>3/2$. Next claim
\begin{equation} \label{uniform_in_h}
\left| \int_{x+2Mi}^{x+iy} g_l(z)dz \right| < c \frac{h_l}{l^{3/2+\epsilon}w^{\kappa}},
\end{equation}
for $-1<x<-1/2$ independent of $y>2M$, and where if $w\ge1$, $\kappa=1$, otherwise, for $w\in(0,1)$, $\kappa = 2$. To prove this, we'll use the case bounds, initially taking $y\to \infty$. Since $w$ can be any value $>0$, the saddle point $2\pi l w^2 i$ can be anywhere in relation to the lower integral limit, and we must consider all case bounds. For case I, we had $\mathcal{O}(h_l/(l^{5/2} w^3))$. For Case II, $\mathcal{O}(h_l/(l^2 w^2))$. Case III is the most restrictive case, namely $\mathcal{O}(h_l/(l^{1-x}w^{-2x}))$, when $-1<x<-1/2$. The worst case is when $x = -1/2-\epsilon$, giving the bound $\mathcal{O}(h_l/l^{3/2+\epsilon}w^{\kappa})$, where the conditions on $\kappa$ account for bounds on $w^{-2x}$ over all $x$ in the range. Now noting the case bounds were constructed by using the absolute value of the integrand, the integral in $(\ref{uniform_in_h})$ for finite $y$ would have a strictly lower value than those considered in the case bounds. This proves the claim. And since the series in $l$ converges over the range of interest for $x$, we have that the sum of integrals in $(\ref{I_sum_over_j})$ converges in this range.

From the claim, we have that the sum,
\[
\sum_{l=1}^{\infty} \int_{x+2Mi}^{x+iy}g_l(z)dz,
\]
converges uniformly in $y$. This suffices to justify swapping the sum and integrals in $(\ref{I_sum_over_j})$, as follows. Consider,

\begin{description}

	\item[i]
	\[
	\left| \sum_{l=1}^{\infty} \int_{x+2Mi}^{x+iy}g_l(z)dz- \sum_{l=1}^{N} \int_{x+2Mi}^{x+iy} g_l(z)dz\right| < \epsilon/3
	\]

	\item[ii]
	\[
	\left| \sum_{l=1}^{N} \int_{x+2Mi}^{x+i\infty}g_l(z)dz- \sum_{l=1}^{\infty} \int_{x+2Mi}^{x+i\infty} g_l(z)dz\right| < \epsilon/3
	\]	
	
	\item[iii]
	\[
	\left| \sum_{l=1}^{N} \int_{x+2Mi}^{x+iy}g_l(z)dz- \sum_{l=1}^{N} \int_{x+2Mi}^{x+i\infty} g_l(z)dz\right| < \epsilon/3
	\]

\end{description}

The inequality in (i) follows since,
\[
\left| \sum_{l=N+1}^{\infty} \int_{x+2Mi}^{x+iy}g_l(z)dz \right| \le \left|\int_{x+2Mi}^{x+iy}g_{N+1}(z)dz \right| +... < c\frac{h_{N+1}}{(N+1)^{3/2}w^{\kappa}} + c\frac{h_{N+2}}{(N+2)^{3/2}w^{\kappa}} + ....
\]
and we can choose $N$ sufficiently large that the sum from $N$ to $\infty$ of the bounds on the summands (from $(\ref{uniform_in_h})$) meet the $\epsilon$ condition, \emph{independent} of $y$. By the same reasoning, allowing $y \to \infty$, the inequality in (ii) is obtained.

Let $N$ be fixed from inequalities (i) and (ii). The inequality in (iii) follows by choosing $y_0$ large enough that, as guaranteed by the convergence of each of the $g_l$ integrals, for all $y>y_0$ and for all $l\in[1,N]$,
\[
\left| \int_{x+2Mi}^{x+iy} g_l(z)dz - \int_{x+2Mi}^{x+i\infty} g_l(z)dz \right| < \epsilon / 3N.
\]
Since the inequality holds for all $l \in [1,N]$, the sum, $\sum_1^N$ must be less than $\epsilon/3$, as desired.

This justifies the inequalities. Transitively they give, for all $y$,
\[
\left| \sum_{l=1}^{\infty} \int_{x+2Mi}^{x+iy}g_l(z)dz - \sum_{l=1}^{\infty} \int_{x+2Mi}^{x+i\infty} g_l(z)dz\right| < \epsilon.
\]
Combined with the uniform convergence of $g_l$, which allows the exchange of order between finite integrals and infinite sums, this proves the sum of integrals, $(\ref{I_sum_over_j})$, equals the integral in $(\ref{I_remainder_w_order_bound})$.

	For the order bound, note for each $l$, the portion of the contour between $x-2Mi$ and $x+2Mi$ has been shown $\mathcal{O} \left(h_l / (l^2 w^{3/2-\epsilon}) \right)$ by $(\ref{middle_portion_hmj})$ (with $x=-3/4+\epsilon$). For the $|y|\ge2M$ portion, the overall bound from the case bounds, as in $(\ref{uniform_in_h})$, gives $\mathcal{O}(h_l / (l^{1+3/4} w^{3/2-\epsilon}))$ for this $x$. This gives, at worst, for the sum of the whole contours along $x=-3/4+\epsilon$, a bound of
\[
\sum_{n=1}^{\infty} c_1\frac{h_n}{l^{1+3/4}w^{3/2-\epsilon}} = c_2w^{-3/2+\epsilon},
\]
proving the theorem.
\end{proof}

\subsection{\bf Summary of $\mathbf{I_{j,k}(w)}$ residue form.} \label{remainder_integral_summary_section}

Combining the results on the remainder integral for the case $j/k \ne 1/2$, $(\ref{non_1_2_order_bound})$ with those for $j/k = 1/2$, $(\ref{I_remainder_w_order_bound})$, shows the validity of equation $(\ref{I_residue})$ for all $j$ and $k$, $k > 1$, $\gcd(j,k)=1$. Equation $(\ref{I_residue})$ can be rewritten as,
\begin{equation} \begin{split} \label{I_residue_full}
I_{j,k}(w) = &\frac{\zeta(1+k/j)}{\zeta(k)} w^{k/j} +  \lim_{m\to\infty}\stackrel{*}{\sum_{|\mbox{Im }\rho|<Y_m/(jk)}} \frac{\zeta(\rho/j+1)\zeta(\rho/k)} {\zeta ^\prime (\rho)}\frac{{ w}^{\rho/j}}{\rho} + \ln w \\
& + \gamma - j\left( 1- \frac{1}{k} \right) \ln 2\pi + \mathcal{O}\left(w^{-3/(2j)+\epsilon}\right)
\end{split} \end{equation}
(* indicates this form for the summands is valid for simple zeta zeros only; note also the remainder $\mathcal{O}$ bound has not been shown to be uniform over $j$ or $k$). This proves the first part of Theorem $\ref{S_jk_theorem}$.

\subsection{\bf The general sum-of-$\mathbf{Q}$'s form, and a first order error term; proof of Theorem $\mathbf{\ref{S_jk_theorem}}$.} \label{sum_of_Qs_section}
\label{I_error_bound_section}

Rewriting the residue sums for $I_{j,k}(w)$ (equation $(\ref{I_residue_full})$) as,
\begin{equation}  \label{I_j_k_w_R}
I_{j,k}(w) =  \frac{\zeta(1+k/j)}{\zeta(k)} {w}^{k/j} + R_{j,k}(w),
\end{equation}
it is possible to obtain overall estimates on the error term, $R_{j,k}(w)$.

To find a $\mathcal{O}$ bound on $R_{j,k}(w)$, we'll use equation $(\ref{I_jk_as_Q_sums})$:
\[
I_{j,k}(w) = \sum_{m=1}^{\lfloor w \rfloor} \frac{1}{m} Q_k\left[ \left( \frac{w}{m} \right)^{k/j} \right],
\]
and denote the order of the error term in $Q_k(x)$ by $\beta_k$. That is,
\begin{equation} \label{Q_k_Oh_bound}
Q_k(x) = \frac{x}{\zeta(k)} + R_{Q_k}(x) = \frac{x}{\zeta(k)} + \mathcal{O}(x^{\beta_k}).
\end{equation}

There has been a lot of work in the literature on finding a close upper bound on $R_{Q_k}(x)$ (see, for example, \cite{Pappalardi} for a summary). It is known that $\beta_k = 1/k$ at worst, and if the Riemann hypothesis is true it is conjectured that $\beta_k = 1/(2k) + \epsilon$ (see \cite{Pappalardi}). The case $k=2$, where $Q_2(x)$ is the square-free counting function, has received special attention. The best known value for $\beta_2$ as of this writing is $17/54+\epsilon$ (see \cite{Jia}).

Substituting $(\ref{Q_k_Oh_bound})$ into equation $(\ref{I_jk_as_Q_sums})$ produces,

\begin{equation} \label{I_as_m_sums}
I_{j,k}(w) = \sum_{m=1}^{\lfloor w \rfloor} \frac{1}{m} \left[ \frac{w^{k/j}}{ \zeta(k) m^{1+k/j} } + \mathcal{O} \left( \left( \frac{w}{m} \right)^{(k/j) \beta_k} \right) \right] = \sum_{m=1}^{\lfloor w \rfloor} \frac{w^{k/j}}{\zeta(k) m^{1+k/j}} +  \sum_{m=1}^{\lfloor w \rfloor} \frac{1}{m} \mathcal{O} \left( \left( \frac{w}{m} \right)^{(k/j) \beta_k} \right).
\end{equation}

Looking at the second sum, by definition of the $\mathcal{O}$ bound we have
\[
\left| \sum_{m=1}^{\lfloor w \rfloor} \frac{1}{m} \mathcal{O} \left( \left( \frac{w}{m} \right)^{(k/j) \beta_k} \right) \right| < \sum_{m=1}^{\lfloor w \rfloor} \frac{1}{m} K \left( \frac{w}{m} \right)^{(k/j) \beta_k} < K\zeta((k/j)\beta_k+1) w^{(k/j)\beta_k} = \mathcal{O}(w^{(k/j)\beta_k})
\]
for some sufficiently large $K$, where the second inequality is conditional on $\beta>0$.

Setting equations $(\ref{I_j_k_w_R})$ and $(\ref{I_as_m_sums})$ equal we obtain
\begin{equation} \begin{split} \label{R_jk_1}
R_{j,k}(w) &= \sum_{m=1}^{\lfloor w \rfloor} \frac{w^{k/j}}{\zeta(k) m^{1+k/j}} + \mathcal{O}(w^{(k/j)\beta_k}) - \frac{\zeta(1+k/j)}{\zeta(k)} w^{k/j} \\
&= -\left( \zeta(1+k/j) - \sum_{m=1}^{\lfloor w \rfloor} \frac{1}{m^{1+k/j}} \right) \frac{w^{k/j}}{\zeta(k)} + \mathcal{O}(w^{(k/j) \beta}).
\end{split} \end{equation}
But by known order bounds on the partial sums of zeta Dirichlet series (see, for example, \cite{Titchmarsh} Theorem 4.11), 
\begin{equation} \label{titch_4_11}
\zeta(1+k/j) - \sum_{m=1}^{w} \frac{1}{m^{1+k/j}} = \frac{j}{k}w^{-k/j} + \mathcal{O}(w^{-1-k/j}),
\end{equation}
to produce, after substituting $(\ref{titch_4_11})$ into equation $(\ref{R_jk_1})$,
\[
-\left( \frac{w^{-k/j}}{(k/j)} + \mathcal{O}(w^{-1-k/j}) \right) \frac{w^{k/j}}{\zeta(k)} + \mathcal{O}(w^{(k/j) \beta_k}) = c_0 + \mathcal{O}(w^{-1}) + \mathcal{O}(w^{(k/j) \beta_k}),
\]
to give,
\begin{equation} \label{R_jk_O_bound}
R_{j,k}(w) = \mathcal{O}(w^{(k/j) \beta_k}).
\end{equation}
This proves the second of the two parts of Theorem $\ref{S_jk_theorem}$, which, in addition to the previous section's proof of the first part, proves the theorem.

\subsection{\bf The particular case of $\mathbf{I_{1,2}(w)}$}

Before getting to the estimates for $S_{j,k}(w)$, we'll compare the estimates of $I_{j,k}(w)$ with its actual values in the particular case of $I_{1,2}(w)$.  From equation $(\ref{I_residue_full})$ we have,
\[
I_{1,2}(w) = \frac{\zeta(3)}{\zeta(2)} {w}^2 + \lim_{m\to\infty}\stackrel{*}{\sum_{|\mbox{Im }\rho|<Y_m/2}} \frac{\zeta(\rho+1)\zeta(\frac{\rho}{2})} {\zeta ^\prime (\rho)}\frac{{ w}^{\rho}}{\rho} + \ln w + \gamma - \frac{1}{2}\ln(2 \pi) + R_I(w)
\]
where $R_I(w) = \mathcal{O}(w^{-3/2+\epsilon})$.

This produces a first order approximation,
\begin{equation} \label{I_first_order}
I_{1,2}(w) \approx \frac{\zeta(3)}{\zeta(2)} {w}^2  \mbox{ when $w\ge 0$, $0$ otherwise}.
\end{equation}

And if we conjecture the contribution from the zeta zero sum has `mean oscillation' zero (for example if $X_t$ is the integral from $0$ to $t$ of the zeta zero sum's contribution, then $\liminf_{t\to\infty}X_t = \mathcal{O}(1)$), then since the oscillatory contribution from the remainder integral, $R_I$, has been shown to be small we have the following `centerline' approximation:
\begin{equation} \label{I_centerline}
I_{1,2}(w) \approx \frac{\zeta(3)}{\zeta(2)} {w}^2 + \chi_{[1,\infty)} \ln w + \gamma - \frac{1}{2}\ln(2 \pi) \mbox{ when $w\ge 0$, $0$ otherwise},
\end{equation}
where $\chi_{[1,\infty)}$ is the indicator function of value $1$ on $[1,\infty)$, and $0$ elsewhere; this serves to truncate the natural logarithm at $w=1$, to avoid large contributions in the region $[0,1)$.

Figures $\ref{I_1_2_figure}$ and $\ref{I_1_2_figure_inset}$ show the first order and centerline estimation curves compared with the actual computed step function. The first order estimate is within 3\% for $w\approx 15$ while the centerline estimate is entwined with the step function itself in this same region.

\begin{figure}[tbp]
\centering
\begin{minipage}{.5\textwidth}
\centering
\includegraphics[width=0.9\textwidth]{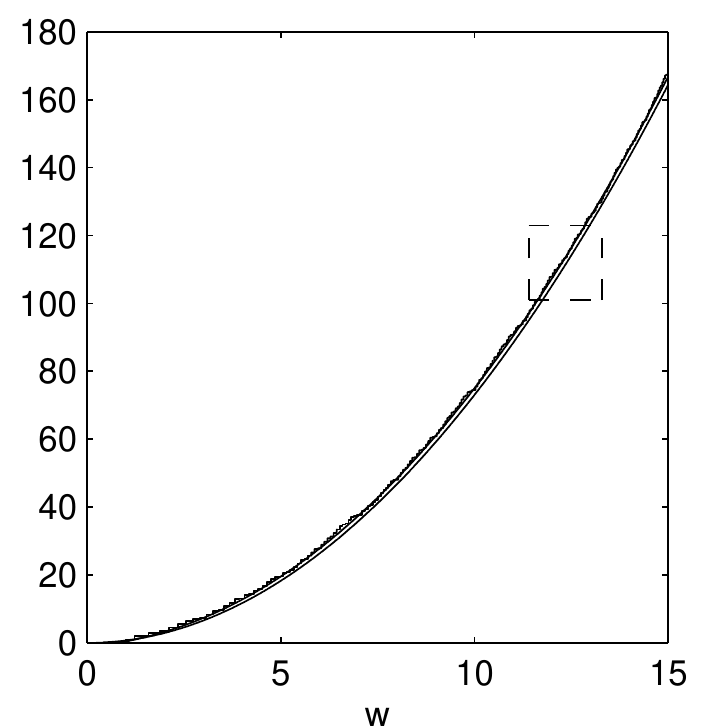}
\caption{\emph{$I_{1,2}(w)$ stairs and estimates.}}
\label{I_1_2_figure}
\end{minipage}\hfill
\begin{minipage}{.5\textwidth}
\centering
\includegraphics[width=0.9\textwidth]{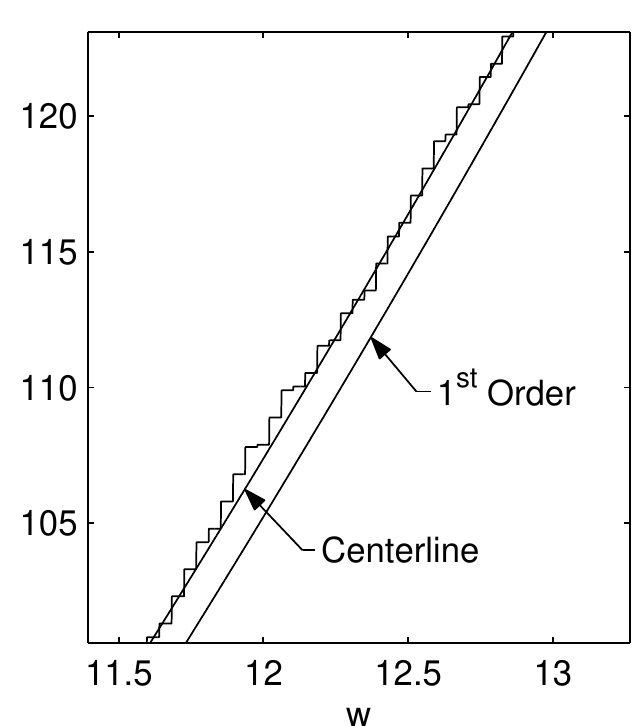}
\caption{\emph{Dashed box inset.}}
\label{I_1_2_figure_inset}
\end{minipage}
\end{figure}

\section{\bf Estimates for $\mathbf{S_{j,k}(w)}$} \label{overall_estimates}

We are now ready to estimate $S_{j,k}(w)$. We'll use the expression of $S_{j,k}(w)$ as the integral of a convolution exponential from equation $(\ref{S_as_conv_exp})$ and the preceding estimates of $I_{j,k}(w)$.

\subsection{\bf Basic results on convolution exponentials.}
In order to develop bounding estimates on $S_{j,k}(w)$ by using estimates on $I_{j,k}(w)$, the following results on convolution exponentials are needed.

\begin{lemma} \label{conv_exp_t_a}
Given $v^a$ supported on $[0,\infty)$, with $a\in(-1,\infty)$,
\[
e^{*c v^a} = \delta_0(t) + \sum_{k=1}^{\infty} \frac{c^k}{k!} \frac{(\Gamma(a+1))^{k}}{\Gamma(k(a+1))} t^{k(a+1)-1},
\]
and the integral, $\int_{t=\alpha}^{w} e^{*c v^a} dt$, where $\alpha$ is $0^+$ or $0^-$, is finite.
\end{lemma}
\begin{proof}

From the definition of the convolution exponential, $e^{*c v^a}[t] = \delta_0(t)+c t^a+(c v^a)^{*2}[t]/2!+...$. Each convolution power $(c v^a)^{*k}$ may be computed from repeated application of the following beta function formula,
\[
\int_0^t v^{q-1}(t-v)^{p-1} dv = t^{p+q-1}B(p,q),
\]
where $B(u,v)$ is the beta function, and with the requirement $p,q>0$ (see for instance \cite{G_R}, section 3.18). For the second convolution power we have,
\[
(c v^a)^{*2} = c^2 \int_{-\infty}^{\infty} \chi_{[0,\infty)}(v) v^a \chi_{[0,\infty)}(t-v) (t-v)^a dv = c^2 \int_0^t v^a (t-v)^a dv = c^2 t^{2a+1} B(a+1,a+1).
\]
Higher powers can be done recursively to give,
\begin{equation} \begin{split} \nonumber
e^{*c v^a} = & \delta_0(t) + \frac{c}{1!} t^a + \frac{c^2}{2!}B(a+1,a+1)t^{2a+1} + \frac{c^3}{3!}B(a+1,a+1)B(a+1,2a+2)t^{3a+2}+... \\
	& +\frac{c^k}{k!}B(a+1,a+1)...B(a+1,(k-1)a+k-1)t^{ka+k-1}+...
\end{split} \end{equation}
Since $B(p,q) = \Gamma(p)\Gamma(q)/\Gamma(p+q)$, the product of $B(a+1,a+1)...B(a+1,(k-1)a+k-1)$ collapses to $(\Gamma(a+1))^{k} / \Gamma(k(a+1))$, proving the series formula.

For the finiteness of the integral, since for $\kappa$ sufficiently large, $k(a+1)-1>0$ for all $k>\kappa$, making $t^{k(a+1)-1}$ monotone increasing in $[0,w]$, we have for some sufficiently large $\kappa' > \kappa$,
\[
\left| \sum_{k=\kappa'}^{\infty} \frac{c^k}{k!} \frac{(\Gamma(a+1))^{k}}{\Gamma(k(a+1))} t^{k(a+1)-1} \right| < \sum_{k=\kappa'}^{\infty} \frac{(|c|\Gamma(a+1))^k w^{k(a+1)-1}}{k! \Gamma(k(a+1))} < \epsilon,
\]
showing the sum converges uniformly in $[0,w]$. This allows the integral to be passed inside. The integral of the first term of the series, the delta function, will be $1$ or $0$, depending on the integral's lower limit. Integrating the remaining sum term by term gives,
\[
\sum_{k=1}^{\infty} \frac{c^k}{k!} \frac{(\Gamma(a+1))^{k}}{\Gamma(k(a+1))} \frac{w^{k(a+1)}}{k(a+1)},
\]
which clearly converges.
\end{proof}

\begin{lemma} \label{convolution_exp_bound}
Given the integral of a convolution exponential,
\[
\int_{t=0^-}^w e^{*f} dt,
\]
where $f$ has support on $[0,\infty)$, $f\ge0$, $f(0)=0$, and $f$ is integrable on any $[0,t_0]$, and suppose $\int_{t=0^-}^w f(t)dt = F(w) < G(w)$ for all $w>0$ where $G\ge 0$, $G$ is differentiable, $G$ has support on $[0,\infty)$, and is allowed to have a jump discontinuity at $0$ of value $c>0$. Also assume the derivative of $G$, call it $g$, obeys $g(t)\ge 0$. Then,
\begin{equation}
\int_{t=0^-}^w e^{*f} dt < \int_{t=0^-}^w e^{*g} dt.
\end{equation}
Under the same conditions except with $G(w)<F(w)$, the inequality may be reversed.
\end{lemma}
\begin{proof}
From the series expansion of $e^{*f(v)}$, we have terms $\int_{t=0^-}^w (f(v))^{*m} / m!$. So it suffices to show 
\begin{equation} \label{convo_exp_bound}
\int_{t=0^-}^w (f(v))^{*m} < \int_{t=0^-}^w (g(v))^{*m}
\end{equation}
for all $m \ge 1$.

The proof will be by induction on $m$. For compactness in notation, drop the dependent variable $v$, and use $\int_*$ as shorthand for $\int_{t=0^-}^w$.

First claim $\int_* f^{*m} = F*f^{*(m-1)}$. This follows since $f(0)=0$, so that $F(0)=0$, and so $\int_* f = F$, and since for any $f_1$ and $f_2$ with support on the half-line, $[0,\infty)$, $f_1*f_2[t]$ will only have support on $[0,t]$. This allows swapping the order of integration in $\int_* \int_{-\infty}^{\infty}f(t-v)f^{*(m-1)}[v] dv$, and the intermediary relation $\int_{0-}^{w} f(t-v)dt = \int_{0-}^{w-v} f(u)du = F(w-v)$, proving the claim.

For $m = 2$ we'll show,
\begin{equation} \label{convo_exp_bound_3}
F*f < G*g  = \int_* g^{*2}.
\end{equation}
By the monotonicity in the convolution operation for positive functions, note whenever $0 \le p(t) < q(t)$ for all $t$ in the support of $p$ and $q$, it is true that $p*p < q*q$ and that $p*r < q*r$ for all positive integrable $r(t)$. This allows,
\[
F*f < G*f = g*F < g*G.
\]
The middle equality follows from $\int_{0-}^{w} f(t-v)dt = F(w-v)$ just above, and the reasoning leading to it (noting as well that $\int_* g = G$ by construction). This proves equation $(\ref{convo_exp_bound_3})$.

Now suppose $(\ref{convo_exp_bound})$ holds at $m-1$, then
\[
\int_* f^{*(m)} = F*f^{*(m-1)} < G*f^{*(m-1)} = g*\left(\int_* f^{*(m-1)} \right) < g*\left( \int_* g^{*(m-1)} \right) = \int_* g^{*(m)},
\]
where the second inequality follows from the induction hypothesis. This completes the proof of $(\ref{convo_exp_bound})$. 

So 
\[
\int_{0-}^{w} e^{*f} < \int_{0-}^{w} \delta(w) + \sum_{m=1}^{\infty} \frac{1}{m!}  \int_{0-}^{w} g^{*m} = \int_{0-}^{w} e^{*g}.
\]
The same steps may be repeated when $G(w)<F(w)$, reversing the inequality signs.
\end{proof}

Finally note the additive property of convolution exponentials, $e^{*(f+g)} = e^{*f}*e^{*g}$, allows the following relation, which will be useful in discussions on estimates of $S_{j,k}(w)$: $e^{*(c\delta(v) + g(v))} = e^{*c\delta}*e^{*g} = e^{c}e^{*g}$ (the second equality follows from $\delta(v)*h(v)[t]=h(t)$, and from the series form for the convolution exponential, noting $(c\delta(v))^{*m} = c^m \delta(v)$).

\subsection{\bf Estimates on $\mathbf{S_{j,k}}$, the general case.}

Recalling that $I_{j,k}$ and $dI_{j,k}$ and all their estimates have support on $[0,\infty)$, we have the following $S_{j,k}$ estimates.

\subsubsection*{\bf First order estimate.}

Combining equations $(\ref{S_as_conv_exp})$ and $(\ref{I_residue_full})$, and noting the derivative of the first order estimate of $I_{j,k}$ in equation $(\ref{I_residue_full})$ produces the form $e^{*c v^a}$ where $a=k/j-1>-1$, Lemma $\ref{conv_exp_t_a}$ can be applied to show the validity of the convolution exponential and the outer integral. The first order approximation of $S_{j,k}(w)$ then is,
\[
S_{j,k}(w) \approx \int_{t=0^+}^{w} \exp^{*}\left( \frac{k}{j} \frac{\zeta(1+k/j)}{\zeta(k)} w^{k/j-1} \right) dt.
\]

\subsubsection*{\bf Centerline estimate.}

Though again the justification is conjectural (as in the discussion leading to $(\ref{I_centerline})$) the following centerline estimate is given (where to obtain $dI$ from $(\ref{I_residue_full})$, we've set $d/dw(\ln w)=0$ in $w\in[0,1]$, and, requiring the integrand to be zero at the lower integral limit, we've treated the constant term as a jump discontinuity, giving a $(\gamma - (1-j/k)\ln(2\pi))\delta_0$ term in $dI$; it emerges as a constant of scale, for the reasons mentioned just below the proof of Lemma $\ref{convolution_exp_bound}$):

\[
S_{j,k}(w) \approx e^{\gamma - j(1-1/k) \ln 2\pi } \int_{t=0^+}^{w} \exp^{*}\left( \frac{k}{j} \frac{\zeta(1+k/j)}{\zeta(k)} w^{k/j-1} + \chi_{[1,\infty)} 1/w \right) dt,
\]
Lemma $\ref{conv_exp_t_a}$ has been applied again as follows: note the argument of the convolution exponential can be bounded by $c_0 w^{k/j-1}$ for some $c_0>0$, for all $w\in[0,\infty)$; Lemma $\ref{conv_exp_t_a}$ is then invoked with $cv^a=c_0v^{k/j-1}$; by monotonicity of convolution powers for positive functions (if $f>g$ then $f^{*k} > g^{*k}$), the convolution exponential and its integral are then finite.

\subsubsection*{\bf Order bounds.}

$S_{j,k}(w)$ obeys the following $\mathcal{O}$ bound,
\[
S_{j,k}(w) = \int_{t=0^-}^{w} \exp^{*}\left( \frac{k}{j} \frac{\zeta(1+k/j)}{\zeta(k)} w^{k/j-1} + \frac{1}{j} \mathcal{O}(w^{1/j-1})\right) dt - 1.
\]
First, we've used the error term formulation of $I_{j,k}$ in $(\ref{I_j_k_w_R})$, where the error term is at worst of order $w^{1/j}$. Lemma $\ref{convolution_exp_bound}$ is then applied as follows. From the definition of an order bound, $I_{j,k}(w)$ can at worst be bounded below by $\max \{H(w) - c_l w^{1/j} - b_l, 0 \}$, and above by $H(w) + c_h w^{1/j} + b_h$, where $H(w)=\zeta(1+k/j)w^{k/j}/\zeta(k)$. Since $k\ge 2$, both the lower and upper bounds are then monotone increasing, making both bounds' derivatives $\ge 0$, completing the requirements of the lemma. (Note also the constant term, $b_h$, when differentiating the bound, will produce $b_h \delta_0$ in the derivative, which has been accounted for in the lemma; there is no jump discontinuity for the lower bound, since $H(0)=0$, and the bound is not allowed to go below $0$.) To prove finiteness of the resulting convolution exponential and its integral, apply Lemma $\ref{conv_exp_t_a}$ by noting that for any fixed $w>0$, there is some $c_w>0$ such that,
\[
c_w v^{1/j-1} \ge v^{k/j-1}+\frac{c}{j}v^{1/j-1} \ge 0,
\]
for all $v\in[0,w]$. We can then apply the lemma with $cv^a=c_w v^{1/j-1}$. The resulting finite bound, $e^{*c_w v^{1/j-1}}$, then applies by monotonicity of convolution powers for positive functions as in the centerline estimate.

\subsection{\bf Estimates in the particular case of $\mathbf{S_{1,2}(w)}$.}

We'll examine some of the differences between the estimates and the actual step function in the case of the sum of square roots counting function, $S_{1,2}(w)$.

\subsubsection*{\bf First order estimate of $\mathbf{S_{1,2}(w)}$.}

In the case of sums of square roots, a simple first order estimate can be derived by using the first order estimate of $I_{1,2}(w)$ from equation $(\ref{I_first_order})$. Lemma $\ref{conv_exp_t_a}$ gives,
\[
e^{*\chi_{[0,\infty)} cw} = \delta(w) + \sum_{m=1}^{\infty} c^m \frac{w^{2m-1}}{(2m-1)! m!}.
\]
This produces
\begin{equation} \label{S_first_order}
S_{1,2}(w) = \int_{t=0^+}^{w} e^{*dI_{1,2}} dt \approx \int_{t=0^+}^{w} e^{*2t\zeta(3) / \zeta(2)} dt =  \sum_{m=1}^{\infty} \left(\frac{2\zeta(3)}{\zeta(2)} \right)^m \frac{w^{2m}}{(2m)! m!}.
\end{equation}

\subsubsection*{\bf Centerline estimate of $\mathbf{S_{1,2}(w)}$.}

Because of the sensitivity of $S_{1,2}(w)$ to errors in estimating $I_{1,2}(w)$, the first order estimate of $S_{1,2}(w)$ can be improved (at least for values of $w$ up to around $30$) by using the so-called centerline estimate for $I_{1,2}(w)$ above (equation $(\ref{I_centerline})$). It gives for $dI_{1,2}$,
\[
dI_{1,2} \approx \frac{2\zeta(3)}{\zeta(2)} w + \chi_{[1,\infty)} \frac{1}{w} + \left( \gamma - \frac{1}{2}\ln(2 \pi) \right) \delta(w) .
\]

This produces,
\[ 
S_{1,2}(w) \approx \int_{t=0^+}^{w} \exp^{*} \left( \frac{2\zeta(3)}{\zeta(2)}w + \chi_{[1,\infty)} \frac{1}{w} + (\gamma - \frac{1}{2}\ln(2 \pi)) \delta(w) \right) dt 
\]
\begin{equation} \label{S_centerline}
= e^{\gamma - \frac{1}{2}\ln(2 \pi)} \left( \sum_{n=1}^{\infty} \left(\frac{2\zeta(3)}{\zeta(2)} \right)^n \frac{w^{2n}}{(2n)! n!} \right) * \left(  e^{*\chi_{[1,\infty)} 1/w} \right).
\end{equation}
The righthand factor, $e^{*\chi_{[1,\infty)} 1/w}[t]$, appears to have no analytical expression, though it can be closely bounded between constants using Stirling numbers of the first kind and the point mass approximation: $\sum_{n=1} \delta(w-n)/n $; numerical tests suggest it is quickly asymptotic (by $t\approx3$) to $\approx 0.59$.

\subsubsection*{\bf Bounds on $\mathbf{S_{1,2}(w)}$.}

The $\mathcal{O}$ bound on $I_{1,2}(w)$ from equation $(\ref{R_jk_O_bound})$ with the help of Lemma $\ref{convolution_exp_bound}$ can produce order bounds on $S_{1,2}(w)$:

\[
S_{1,2}(w) = \left( 1 + \sum_{n=1}^{\infty} \left(\frac{2\zeta(3)}{\zeta(2)} \right)^n \frac{w^{2n}}{(2n)! n!} \right) * \left(e^{*\mathcal{O} (w^{2\beta_2-1})} \right) - 1.
\]

\subsubsection*{\bf Discussion of estimates of $\mathbf{S_{1,2}(w)}$.}

Figures $\ref{S_estimates}$ and $\ref{S_estimates_inset}$ show comparisons between the actual curve of $S_{1,2}(w)$ and the estimates of equations $(\ref{S_first_order})$ and $(\ref{S_centerline})$. These curves show the rapid increase of $S_{1,2}(w)$ as it exceeds the order of any polynomial but remains sub-exponential.

The centerline estimate stays within 10\% of the actual curve while the first order estimate shows a wider variation. 

\begin{figure}[tbp]
\centering
\begin{minipage}{.5\textwidth}
\centering
\includegraphics[width=0.9\textwidth]{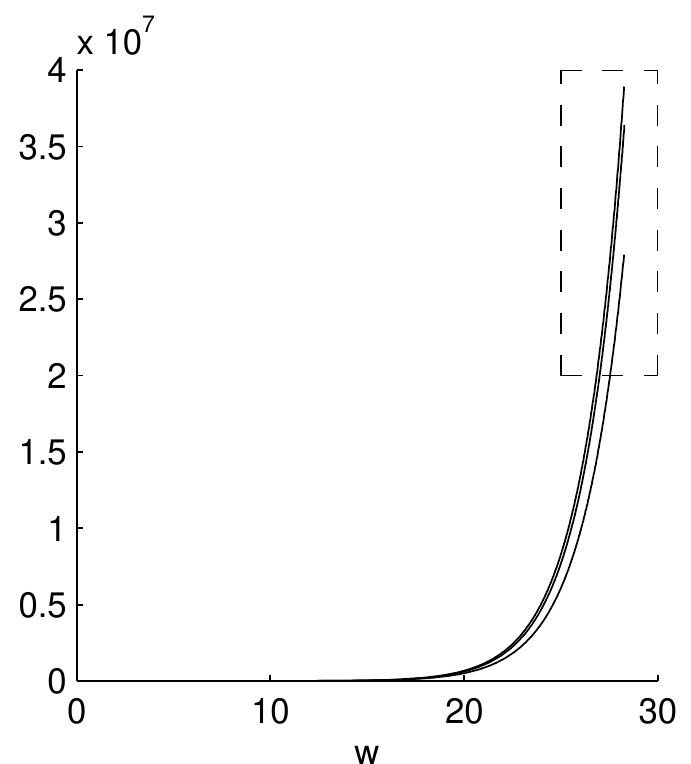}
\caption{\emph{Estimates of $S_{1,2}(w)$.}}
\label{S_estimates}
\end{minipage}\hfill
\begin{minipage}{.5\textwidth}
\centering
\includegraphics[width=0.9\textwidth]{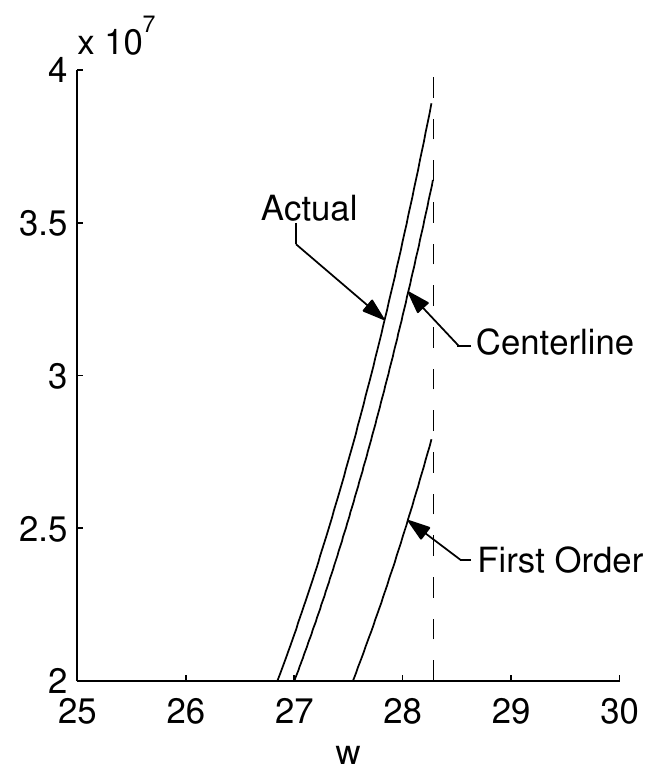}
\caption{\emph{Dashed box inset.}}
\label{S_estimates_inset}
\end{minipage}
\end{figure}

\subsection{\bf Remarks on accuracy in the general case, $\mathbf{S_{j,k}(w)}$.}

Though in general the first order and centerline $I_{j,k}$ estimates initially do not deviate from the $I_{j,k}$ staircase by more than a few percent, the effect of the error is magnified when $dI_{j,k}$ is put through the convolution exponential; consider that shift upward in the estimate of $I_{j,k}$ of magnitude $1$ translates to a point mass of weight $1$ for $dI$'s estimate at $w=0$, which in turn results in an additional factor of $e^1$ in the estimate of $S_{j,k}$ (to help see why, see just below the proof of Lemma $\ref{convolution_exp_bound}$). In other words, estimates for $S_{j,k}$ are very sensitive to errors in the $I_{j,k}$ estimates; offset errors in the second translate into errors of \emph{scale} for the first. From numerical checks when $\{j,k\} \ne \{1,2\}$, both the first order and centerline estimates for example can be quite poor in terms of relative error. 

If greater accuracy is required in estimating $S_{j,k}$, one approach that has worked in practice are hybrid estimates to $I_{j,k}$. One involves using the first order estimate in $[0,w_0]$ for some suitable $w_0$, then the centerline estimate is used for the remainder of the interval of support. (Another improvement may result from truncating the $I_{j,k}$ estimate below $w=1$, so that both the estimate to $I$ and its derivative are zero in $[0,1]$.) Another form of hybrid estimate would be to hard code the actual $I_{j,k}$ and $dI_{j,k}$ function up to some threshold $w_0$, then resume with the first order or centerline estimates for the rest, since on a fixed interval $t\in[0,w]$, with $f(t)$ supported in $[0,\infty)$, the convolution exponential of $f$ is especially sensitive to errors in $f$ at low values of $t$.

\noindent\hrulefill

\appendix
\section*{\bf Appendix}
\renewcommand{\thesection}{A}

\subsection*{\bf Case bounds.}
This section is devoted to finding order bounds in $l$ (and the accompanying $w$ order) on the following integral (equation $(\ref{jth_summand_asymptote})$):
\[
\frac{ h_l }{ l } \frac{1}{2\sqrt{2 \pi}} \int_{x+y_0i}^{x+i\infty} e^{i\frac{\pi}{2}z} z^{-(z+\frac{3}{2})} (2\pi l w^2 e)^z dz,
\]
where $l\ge1$ and $y_0 \ge 2M = 3$ and the requirement of $-3/2 < x < -1/2$ (from the requirement on the expansion of the negative tri-zeta term (equation $(\ref{hmj_v_zeta})$) and from the low $x$ limit for convergence of vertical contours (equation $(\ref{I_remainder_bel_1_k})$) ). 

The method for a given $l$ and $y_0$ is to take the absolute value of the integrand and follow the path(s) of steepest descent, which allows finding a tight bound while avoiding the complications that can arise from oscillatory terms. The Cauchy integral theorem then guarantees the equivalence between the original vertical contour of  $(\ref{jth_summand_asymptote})$  and the steepest descent paths since the integrand has no off-axis singularities in the negative-real half-plane when $|\mbox{Im }z|>0$. 

Recall the location of the saddle point in the $+\mbox{Im }$ half-plane: $i\alpha_l = 2\pi i l w^2$. Under the scaling $s = z/ (l w^2)$ equation $(\ref{jth_summand_asymptote})$ became
\[
\frac{ h_l }{ l^{\frac{3}{2}} w} \frac{1}{2\sqrt{2 \pi}} \int_{s=\frac{x}{l w^2}+\frac{iy_0}{lw^2}}^{\frac{x}{l w^2}+i\infty} s^{-\frac{3}{2}} \exp\left[ l w^2 s \left( \frac{i\pi}{2} + 1 + \ln \left(\frac{2\pi}{s} \right) \right) \right] ds.
\]
Rewriting the exponential portion as $e^{lw^2(u(a,b) + iv(a,b))}$, the negative of the gradient of $u(a,b)$ is shown superimposed on the complex s-plane in Figure $\ref{SD_Case_contours}$. The saddle point is shown at $2\pi i$ along with the proposed (scaled) contours that will be used below. The contours have been chosen according to the method of steepest descent to minimize the error in estimating bounds for equation $(\ref{jth_summand_asymptote})$ when evaluating it by taking the absolute value of its integrand.

\begin{figure}[h]
\centerline{\includegraphics{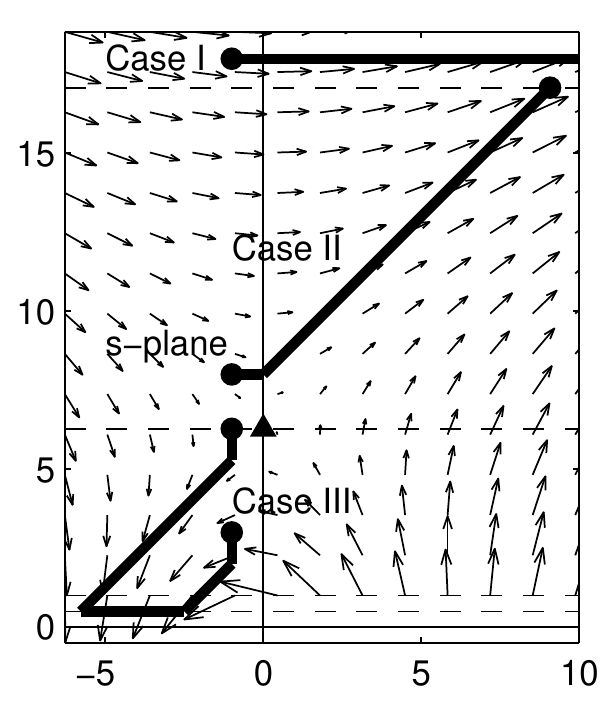}}
\caption{Contours in the scaled s-plane used for the Case bounds. The vector field depicts the negative gradient of $u(a,b)= a\ln(2\pi) + a(1-\ln(r)) + b(\theta-\pi/2)$. The triangle denotes the saddle point and the horizontal dashed lines are at $iM$, $2iM$, $2\pi i$ and $2\pi e i$.}
\label{SD_Case_contours}
\end{figure}

The following bounds will be derived by considering the contours in Figure $\ref{SD_Case_contours}$. We'll use the absolute value of the integrand in $(\ref{jth_summand_asymptote})$ without the leading $h_l / ( 2\sqrt{2 \pi} l) $ term (for simplicity of notation):
\[
\left| e^{i\frac{\pi}{2}z} z^{-(z+\frac{3}{2})} (2\pi l w^2 e)^z \right| = r^{-3/2} e^{y(\theta - \pi/2)} \left( \frac{\alpha_l e}{r} \right) ^x \equiv F(z)
\]

Also for clarity in notation, denote $x$ in the vertical contour of $(\ref{jth_summand_asymptote})$ by $x_0=x$.

Examining the following three cases of $iy_0$ in reference to saddle point $i\alpha_l$ will provide the needed bounds:
\begin{enumerate}
\item[Case I:] $\alpha_l e \le y_0$
\item[Case II:] $\alpha_l \le y_0 \le \alpha_l e$
\item[Case III:] $y_0 \le \alpha_l$
\end{enumerate}

In all Cases, the following bounds on the arctangent will be useful:
\begin{equation} \label{arctan_bounds}
\frac{u}{v} - \frac{1}{3} \left( \frac{u}{v} \right)^3  < \arctan \left(\frac{u}{v} \right) \le \frac{u}{v} \mbox{when $0\le u \le v$}
\end{equation}
along with the well-known identity $\arctan \left( u/v\right) = \pi/2 - \arctan \left(v/u \right)$

\subsubsection*{\bf Case I}
In this Case, $\alpha_l e < y_0$. As suggested by $u(a,b)$ and Figure $\ref{SD_Case_contours}$, the contour will be changed from the line connecting $x+iy_0$ and $x+i\infty$ to the line connecting $x+iy_0$ and $\infty+iy_0$. (At the end of the Cases below, these horizontal and vertical contours in the original integral, $(\ref{hmj_simpler_integrand})$, will be shown to be equal.)

The horizontal contour will be split into three parts. Part I will be the line segment connecting $x_0+iy_0$ and $iy_0$, Part II will be the segment beteween $iy_0$ and $y_0+iy_0$ and Part III will be the segment connecting $y_0+iy_0$ and $\infty+iy_0$.

Part I: 
By $(\ref{arctan_bounds})$ combined with $e^{y_0(\theta - \pi/2)} = e^{y_0(\arctan(|x|/y_0))} < e^{|x|}$, and since for any $x$ on the contour,
\[
\left( \frac{\alpha_l e}{ r} \right)^{x} \le \left( \frac{\alpha_l e}{ r} \right)^{x_0} < \left( \frac{\alpha_l e}{c y_0} \right)^{x_0}
\]
(where the last inequality follows by choosing $c$ such that $r<cy_0$ for all $y_0>2M=3$ and $x_0\in(-3/2,-1/2)$), 
\[
\int_{x_0+iy_0}^{iy_0} F(z)dz < \int_{x_0+iy_0}^{iy_0} y_0^{-3/2} e^{|x_0|} \left( \frac{\alpha_l e}{c y_0} \right)^{x_0} dz < y_0^{-3/2} e^{|x_0|}  \left( \frac{\alpha_l e}{c y_0} \right)^{x_0} |x_0|.
\]
Writing $y_0 = t \alpha_l e$, $t>1$, the bound becomes,
\begin{equation} \label{Case_I_a}
t^{-3/2-x_0} c^{-x_0} e^{|x_0|} |x_0| (\alpha_l e)^{-3/2} < c^{-x_0} e^{|x_0|} |x_0| (\alpha_l e)^{-3/2}
\end{equation}
where the inequality follows since $-3/2-x_0<0$ and $t>1$.

Part II:
By $(\ref{arctan_bounds})$ again,
\[ \textstyle
e^{y_0(\theta - \pi/2)} < e^{y_0(-x/y_0 + (1/3)(x/y_0)^3)} < e^{-(2/3)x}
\]
where the last inequality results from applying a bounding chord between the contour endpoints at $x=0$ and $x=y_0$ of the curve of positive concavity in this region, $y_0(-x/y_0 + (1/3)(x/y_0)^3)$. Since $r^{-3/2} < y_0^{-3/2}$, and $\left( \alpha_l e /r \right) ^x \le 1$, 
\begin{equation} \label{Case_I_b}
\int_{iy_0}^{y_0+iy_0} F(z) dz < \frac{3}{2} y_0^{-3/2} \left( 1-e^{-(2/3)y_0} \right)
\end{equation}

Part III: 
On this contour, $e^{y(\theta-\pi/2)} \le e^{-y_0 \pi/4}$, $r^{-3/2} < y_0^{-3/2}$ and $\left(\alpha_l e / r\right)^x \le \left(1 / \sqrt{2}\right)^x$ since $\alpha_l e \le y_0$. So,
\[
 \int_{y_0+iy_0}^{\infty+iy_0} F(z) dz <  y_0^{-3/2} e^{-y_0 \pi/4} \int_{x=y_0}^{\infty} \left(\frac{1}{\sqrt{2}}\right)^{x} dz 
\]
\begin{equation} \label{Case_I_c}
< \frac{2}{\ln 2} y_0^{-3/2} e^{-y_0 (\pi/4+\ln(\sqrt{2}))} 
\end{equation}

\subsubsection*{\bf Case II}
In this Case, the integral of $(\ref{jth_summand_asymptote})$ has lower limit $y_0$ between $\alpha_l$ and $\alpha_l e$. The contour also will consist of three parts: the path $x_0+iy_0$ to $iy_0$, the path along a diagonal of slope $+1$ to $\alpha_l e - y_0 + i\alpha_l e$, and the horizontal contour from  $\alpha_l e - y_0 + i\alpha_l e$ to $\infty + i\alpha_l e$.

Part I:
By reworking Case I with $\alpha_l$ in place of $\alpha_l e$, the bound for the section $x_0+iy_0$ to $iy_0$ is $c^{-x_0} e^{|x_0|} |x_0| (\alpha_l)^{-3/2}$.

Part II:

Call this diagonal contour $C_d$. Rewrite $F(z)$ as $r^{-3/2} e^{y(\theta-\pi/2)+x(\ln (\alpha_l e)-\ln r)}$, then noting along $C_d$, $y=x+y_0$ and $r=\sqrt{x^2+(x+y_0)^2}$, and that $\theta = \tan^{-1} (-x/(x+y_0)) + \pi/2$, define
\[
h(x) = (x+y_0)\tan^{-1}\left( \frac{-x}{x+y_0} \right) + x\left[ \ln(\alpha_l e)-\frac{1}{2}(x^2+(x+y_0)^2)\right].
\]
Next factor out the $\alpha_l$ term, and let $v=x/y_0$ to get $h(x) = (\alpha_l/y_0)^x e^{g(v)}$, where,
\[
g(v) = y_0 \left( (1+v)\arctan \left( -v / (1+v) \right) +v \left( 1-(1/2)\ln(v^2+(1+v)^2) \right) \right).
\]

This gives $F(z) = r^{-3/2} \left( \frac{\alpha_l}{y_0} \right)^x e^{g(v)}$. The following bound applies to $g(v)$: $g(v) \le -(y_0 / 2) v^2$ on $v \in [0,e-1]$ (shown at the end of the Cases below). Also, along the contour, $r^{-3/2} \le {y_0}^{-3/2}$ and $\left( \alpha_l / y_0 \right)^x \le 1$. So,
\begin{equation} \label{Case_II_a}
\int_{C_d} F(z) dz < {y_0}^{-3/2} \int_{x=0}^{\alpha_l e-y_0} e^{-x^2/(2y_0)} dx < \sqrt{\frac{\pi}{2}}y_0^{-1}
\end{equation}

Part III:
Lastly, for the portion from $\alpha_l e - y_0 + \alpha_l e i$ to $\infty + \alpha_l e i$, Case I, equations $(\ref{Case_I_b})$ and $(\ref{Case_I_c})$ with $y_0=\alpha_l e$ provide the needed bound since the horizontal contour integral is monotone in its lower limit.

\subsubsection*{\bf Case III}
In this Case, $y_0 < \alpha_l$. The path will be as follows: starting at $x_0+iy_0$, go in the $-\mbox{Im}$ direction to meet the diagonal connecting $iy_0$ and $-y_0$ then follow the diagonal down to $-y_0+M+iM$, then proceed in the $-\mbox{Re}$ direction to $-\alpha_l +M + iM$ then along a return diagonal to $x_0 + i(\alpha_l +x_0)$ and finally end at $x_0+\alpha_l i$ after going in the $+\mbox{Im}$ direction. The bounds from Case II will then complete the contour from $x_0 + \alpha_l i$ to $x_0 + i\infty$.

Part I:
First, to treat the vertical portions of the contour along the lines connecting $x_0 + iy_1$ and $x_0 + i(y_1+x_0)$, where $y_1 \in \{ \alpha_l, y_0\}$, the following bounds apply: $r^{-3/2} < (y_1+x_0)^{-3/2}$, $e^{y(\theta-\pi/2)} < e^{|x_0|}$ (by $(\ref{arctan_bounds})$) and 
\[ \textstyle
\left( \frac{\alpha_l e}{r} \right)^x < \left( \frac{\alpha_l e}{y_1 - x_0} \right)^{x_0}.
\]
This produces,
\begin{equation} \label{Case_III_a}
\int_{x_0+iy_1}^{x_0 + i(y_1+x_0)} F(z) dz < (y_1+x_0)^{-3/2}  e^{|x_0|} \left( \frac{\alpha_l e}{y_1 - x_0} \right)^{x_0} |x_0|.
\end{equation}

Part II:
Next, we will find a bound on the diagonal portions. The procedure is very similar to that of Case II, part II. The contour, call it $C_d$, will be $x_0 + i(y_1+x_0)$ to $-y_1 + M +Mi$ where $y_1=y_0$ for the portion descending to $\mbox{Im }z=M$ and where $y_1=\alpha_l$ for the return. Rewrite the integrand as,
\[
r^{-3/2} e^{y(\theta - \pi/2)} \left( \frac{\alpha_l e}{r} \right) ^x = r^{-3/2} \left( \frac{\alpha_l}{y_1} \right)^x e^{g(v)}
\]
where, similar to Case II, $g(v) = y_1 \left( (1+v)\arctan \left( -v / 1+v \right) +v \left( 1-(1/2)\ln(v^2+(1+v)^2) \right) \right)$ with $v = x / y_1$. Here, $g(v)$ can be bounded: $g(v) \le -y_1 v^2$ on $v \in [-1,0]$ (shown at the end of the Cases below). Also, along the contour, $r^{-3/2} \le \left( y_1 / \sqrt{2}\right)^{-3/2}$ and $\left( \alpha_l / y_1 \right)^x \le \left( \alpha_l / y_1 \right)^{x_0}$. So,
\begin{equation} \label{Case_III_b}
\int_{C_d} F(z) dz < \left(\frac{y_1}{\sqrt{2}}\right)^{-3/2} \left( \frac{\alpha_l}{y_1} \right)^{x_0} \int_{x=x_0}^{-y_1+M} e^{-x^2/y_1} dx < \frac{\sqrt{\pi}}{2^{1/4}} y_1^{-1-x_0} {\alpha_l}^{x_0}
\end{equation}

Part III:
Next, we will find a bound on the horizontal portion from $-y_0+M+Mi$ to $-\alpha_l +M+Mi$; call this contour $C_h$. First, since $\theta \in [3\pi / 4,\pi)$, $e^{y(\theta-\pi/2)} < e^{M \pi/2}$. Next, to bound $\left( \alpha_l e / r \right)^x$, rewrite it as $\exp[x(\ln(\alpha_l e)-\ln r)] = e^{h(x)}$. Since $x$ is at most $-M$ (this occurs in the event $y_0=2M$), take for the interval of interest $I_x: x \in [M-\alpha_l,-M]$. We'll show $h(x)$ is negative and monotone increasing in this interval. This allows bounding the whole term by its value at its rightmost (most positive) possible endpoint. Since $0<r<\alpha_l<\alpha_l e$, $\ln(\alpha_l e / r) > 0$. This proves $h(x)<0$ on the interval in question. Next, take the derivative of $h(x)$:
\[
\frac{d}{dx} h(x) = 1-\frac{x^2}{x^2+M^2} + \ln \left(\frac{\alpha_l}{r} \right).
\]
Since both $1-x^2/(x^2+M^2)$ and the $\ln$ term are greater than zero on $I_x$, this proves $h(x)$ is monotone increasing, as desired. This allows the bound,
\[ \textstyle
\left( \frac{\alpha_l e}{r} \right)^x \le \left( \frac{\alpha_l e}{\sqrt{2}M} \right)^{-M},
\]
valid on $I_x$. The whole integral can then be bounded,
\begin{eqnarray} \label{Case_III_c}
\int_{C_h} F(z) dz &<& e^{M \pi/2} \left( \frac{\alpha_l e}{\sqrt{2}M} \right)^{-M} \int_{C_h} r^{-3/2} dz \nonumber\\
&<& 2 e^{M \pi/2} \left( \frac{\alpha_l e}{\sqrt{2}M} \right)^{-M} \left( \frac{1}{\sqrt{y_0-M}} - \frac{1}{\sqrt{\alpha_l-M}}\right)
\end{eqnarray}
(where, recall, $M = 3 / 2$).

Part IV:
For the remainder of the contour, $x_0 + i \alpha_l$ to $x_0 + i\infty$, the bounds of Case II can be used by setting $y_0 = \alpha_l$.

\subsection*{\bf Supporting bounds.}
This section addresses the open items from the Cases.

\subsubsection*{\bf Horizontal vs. vertical contours.}
In Cases I and II, a bound was developed for a horizontal contour. To show this is equivalent to the original vertical contour in $(\ref{hmj_simpler_integrand})$, we will first show that a contour from $z=K+iy_0$ to $K+iy_1$, $y_0<y_1$ and $K>>0$ approaches zero as $K\to\infty$. Since $e^{y(\theta-\pi/2)} < 1$, $F(z)$, the integrand of the principal term in the asymptotic expansion, $(\ref{asymptote_integrand})$, is less than $(\alpha_l e)^K (K^2+y^2)^{-(1/2)(K+3/2)}$, producing the bound
\[
\int_{K+iy_0}^{K+iy_1} F(z) dz < c(y_1-y_0) \frac{(\alpha_l e)^K}{(K^2+y_0^2)^{(1/2)(K+3/2)}}
\]
which goes to zero in the limit $K\to\infty$. Since the bound on the integrand also holds after including the remainder terms $R_k$ (within the mentioned constant factor), the integral of $(\ref{hmj_simpler_integrand})$ must go to zero for this modified contour in the limit of $K$ as well.

Next, for a contour from $x_0+iy_1$ to $K+iy_1$ where $y_1>\alpha_l e$, Case I shows a bound of $\mathcal{O} \left( y_1^{-3/2-x_0} \right)$ for fixed $j$ and so must vanish as $y_1\to\infty$ for all $K>0$. Since, again, the bound applies equally well when the remainder terms are included, $(\ref{hmj_simpler_integrand})$ must vanish along this contour in the limit $y_1 \rightarrow \infty$.

Since the integrand of $(\ref{hmj_simpler_integrand})$ has no singularities in the $y>0$ half-plane, we can apply the Cauchy integral theorem to the rectangle formed by $x+iy_0$, $K+iy_0$, $K+iy_1$ and $x+iy_1$. Taking the limit in $K$ and $y_1$ shows the equivalence of horizontal and vertical contours in $(\ref{hmj_simpler_integrand})$.

\subsubsection*{\bf The parabolic bounds.}
In Cases II and III, parabolic bounds were used within the integrand for the diagonal portions of the contours. To justify these bounds, recall the function in question was
\[ \textstyle
g(v) = y_1 \left( (1+v)\arctan \left( \frac{-v}{1+v} \right) +v \left( 1-\frac{1}{2}\ln(v^2+(1+v)^2) \right) \right)
\]
and a bound of $-y_1 v^2$ was claimed on $v \in [-1,0]$ and a bound of $-(y_1 / 2) v^2$ was claimed on $v \in [0,e-1]$. For the bound when $v \le 0$, notice the bound and $g(v)$ both are equal to zero and have a derivative of zero at $v=0$ as well as share a point of intersection at $v=-1$. Two applications of the mean value theorem will show for there to be another point of intersection in the interior $v \in (-1,0)$, the second derivatives of $g(v)$ and the bound must agree in at least two points within the interval. Since $g^{\prime \prime}(v) = -2y_1(v+1) / (v^2+(1+v)^2)$, it is easy to verify this does not occur. Checking an arbitrary point in the interior that the bound is greater than $g(v)$ completes the proof. Similar reasoning applies for the bound when $v \ge 0$, noting that a point of intersection occurs within $v \in (e-1,4)$.

\end{document}